\documentclass[a4paper,11pt]{article}

\usepackage[english]{babel}
\usepackage[english]{translator}
\usepackage[T1]{fontenc}
\usepackage[utf8]{inputenc}
\usepackage{amsmath}
\usepackage{amssymb}
\usepackage{amsfonts}
\usepackage{amstext}
\usepackage{amsthm}
\usepackage{mathtools}
\usepackage{bbm}
\usepackage{graphicx}
\usepackage{xcolor}
\usepackage{psfrag}
\usepackage{pgfplots}
\usepackage{adjustbox}
\usepackage{subfig}
\usepackage{colortbl}
\usepackage{titlesec}
\usepackage{listings}
\usepackage{enumerate}
\usepackage{enumitem}
\usepackage{todonotes}
\definecolor{asparagus}{rgb}{0.53, 0.66, 0.42}
\definecolor{ballblue}{rgb}{0.13, 0.67, 0.8}
\definecolor{cadmiumgreen}{rgb}{0.0, 0.42, 0.24}
\definecolor{cobalt}{rgb}{0.0, 0.28, 0.67}
\definecolor{darklavender}{rgb}{0.45, 0.31, 0.59}
\definecolor{green(pigment)}{rgb}{0.0, 0.65, 0.31}

\usepackage{hyperref}
\usepackage{geometry}
\usepackage{tabularx}
\makeatletter
\def\hlinewd#1{%
\noalign{\ifnum0=`}\fi\hrule \@height #1 %
\futurelet\reserved@a\@xhline}
\makeatother

\hypersetup{
  pdffitwindow=false,
  pdfhighlight=/O,
  pdfnewwindow,
  colorlinks=true,
  citecolor=red,       
  linkcolor=blue,          
  menucolor=blue,        
  urlcolor=blue,            
  pdfpagemode=UseOutlines,
  bookmarksnumbered=true,
  linktocpage,
  pdfkeywords={},
  pdfcreator={pdflatex},
  pdfproducer={LaTeX with hyperref}
}

\newcommand{\C}{\mathbb{C}}               
\newcommand{\R}{\mathbb{R}}               
\renewcommand{\Re}{\mathrm{Re}\,}

\renewcommand{\ker}{\mathrm{ker}}
 
\newcommand{\ci}{\mathrm{i}}

\newcommand{\LOD}{{\text{\normalfont ms}}}

\newcommand{\Rlod}{P_h^{\LOD}}
\newcommand{\Rlodorth}{P_h^{\perp,\LOD}} %
\newcommand{\Vlod}{V_{h}^{\LOD}}

\renewcommand{\div}{\mathrm{div} \,}		

\newcommand{\Cor}{\mathcal{C}}
\newcommand{\ulod}{u_h^{\LOD}}
\newcommand{\vlod}{v_h^{\LOD}}
\newcommand{\bfA}{\boldsymbol{A}}

\newcommand{\RitzLOD}{ R_{h}^{\LOD}}

\newcommand{\bfH}{\boldsymbol{H}}
\renewcommand{\div}{\mathrm{div} \,}	
\newcommand{\curl}{\mathrm{curl} \,}
\newcommand{\ucol}{u}
\newcommand{\Ecol}{E}

 \newcommand{\Csol}{\rho(\kappa)} 
\newcommand{\Csolinv}{\rho(\kappa)^{-1}} 
\newcommand{\orthiu}{( \ci \ucol)^{\perp}} 

\renewcommand{\Re}{\mathrm{Re}\,}

\newcommand{\Ph}{\pi_{h}} 
\newcommand{\EGL}{E_{\mbox{\tiny GL}}} 
\newcommand{\dx}{\,\,\textnormal{d}x}

\newcommand{\Ku}{K_{\tau,\kappa}(u)}

\newcommand{\quotes}[1]{``#1''}

\theoremstyle{definition}
\newtheorem{definition}{Definition}[section]

\theoremstyle{plain}
\newtheorem{theorem}[definition]{Theorem}
\newtheorem{lemma}[definition]{Lemma}
\newtheorem{corollary}[definition]{Corollary}

\allowdisplaybreaks

\begin{document}

\begin{center}
{\Large 
The Ginzburg-Landau equations: Vortex states and numerical multiscale approximations\renewcommand{\thefootnote}{\fnsymbol{footnote}}\setcounter{footnote}{0}
 \hspace{-3pt}\footnote{
 Both authors acknowledge funding by the Deutsche Forschungsgemeinschaft (DFG, German Research Foundation). C.~D\"oding  was supported under Germany's Excellence Strategy -- EXC-2047/1 -- 390685813 and P.~Henning through the DFG project grant HE 2464/7-1 -- 496389671.}}
\end{center}

\begin{center}
{\large Christian D\"oding\footnote[1]{Institute for Numerical Simulation, University Bonn, DE-53115 Bonn, Germany, \\ e-mail: \textcolor{blue}{doeding@ins.uni-bonn.de}.}} and 
{\large Patrick Henning\footnote[2]{Department of Mathematics, Ruhr University Bochum, DE-44801 Bochum, Germany, \\ e-mail: \textcolor{blue}{patrick.henning@rub.de}.}}\\[2em]
\end{center}

\noindent
\begin{center}
\begin{minipage}{0.8\textwidth}
  {\small
    \textbf{Abstract.} 
    In this review article, we provide an overview of recent advances in the numerical approximation of minimizers of the Ginzburg--Landau energy in multiscale spaces. Such minimizers represent the most stable states of type-II superconductors and, for large material parameters~$\kappa$, capture the formation of lattices of quantized vortices. As the vortex cores shrink with increasing~$\kappa$, while their number grows, it is essential to understand how~$\kappa$ should couple to the mesh size in order to correctly resolve the vortex patterns in numerical simulations. We summarize and discuss recent developments based on LOD (Localized Orthogonal Decomposition) multiscale methods and review the corresponding error estimates that explicitly reflect the $\kappa$-dependence and the observed superconvergence. In addition, we include several minor refinements and extensions of existing results by incorporating techniques from recent contributions to the field. Finally, numerical experiments are presented to illustrate and support the theoretical findings. 
}
\end{minipage}
\end{center}

\section{Introduction}

In most materials, the flow of electric current encounters resistance, leading to energy dissipation. Superconductors, however, form a remarkable class of materials that exhibit zero electrical resistance, thereby enabling numerous technological applications.

The standard mathematical framework for superconductivity is the Ginzburg--Landau (GL) model \cite{DuGP93,Landau,SaS07,Tinkham}. Let $\Omega \subset \R^d$ with $d=2$ or $3$ denote the region occupied by the superconductor. The superconducting state is described by a complex-valued order parameter $u \colon \Omega \to \C$. While $u$ itself is not directly observable, its modulus squared $|u|^2$ represents the density of superconducting electron pairs (Cooper pairs) and can be measured experimentally.

Typically, the order parameter satisfies $0 \leq |u|^2 \leq 1$. At points $x \in \Omega$ where $|u(x)|^2=0$, the material is in the normal (non-superconducting) state, whereas $|u(x)|^2=1$ corresponds to the fully superconducting state. Intermediate values $0 < |u(x)|^2 < 1$ describe a mixed state, where superconducting and normal phases coexist with varying density of Cooper pairs.

In such mixed regimes, the material may form structures known as Abrikosov vortex lattices \cite{Abr04}, in which the order parameter vanishes at the vortex centers. These configurations are characteristic of type-II superconductors, which allow partial penetration of an external magnetic field $\mathbf{H}$ provided its strength lies within an appropriate range.

In this setting, both the unknown order parameter $u$ and the internal magnetic field $\curl \bfA$ are determined as minimizers of the \textit{Ginzburg-Landau free energy} (cf.~\cite[Sec.~3]{DGP92}),
\begin{eqnarray} \label{eq:energy_functional}
	\EGL(u,\bfA) 
	&:=&
	\frac{1}{2} \int_\Omega 
	\Bigl| \frac{\ci}{\kappa} \nabla u + \bfA u \Bigr|^2 
	+ 
	\frac{1}{2} \bigl( 1- |u|^2 \bigr)^2
	+
	|\curl \bfA - \bfH |^2
	\dx,
\end{eqnarray}
where $\bfH$ is the applied external magnetic field, $\kappa >0$ is the Ginzburg--Landau parameter, and $\bfA$ denotes the magnetic vector potential. The relation $\bfH_{\text{int}} = \curl \bfA$ provides the internal magnetic field of the superconductor.

Each term in \eqref{eq:energy_functional} reflects a competing physical mechanism: the first term represents kinetic energy, the second enforces the superconducting preference $|u|^2 = 1$, and the third favors alignment of $\bfH_{\text{int}} = \curl \bfA$ with the applied external field $\bfH$. These contributions are mutually conflicting, except in the trivial case $\bfH = 0$, where $(u,\bfA) = (1,\mathbf{0})$ yields zero energy. In practice, minimizers balance these competing tendencies.

The most intriguing states are the vortex lattices that arise in type-II superconductors. Their existence and structure depend crucially on $\kappa$. For small $\kappa$, no vortices occur, while for sufficiently large $\kappa$, vortex lattices emerge. As $\kappa$ increases, the number of vortices grows while their size shrinks, leading to finer lattice patterns. Figure~\ref{fig:vortices_reference} illustrates this phenomenon where densities $|u|^2$ of minimizing order parameters are shown in a two dimensional model configuration. 

\begin{figure}[h]
\centering
\begin{minipage}{0.19\textwidth}
\includegraphics[scale=0.14]{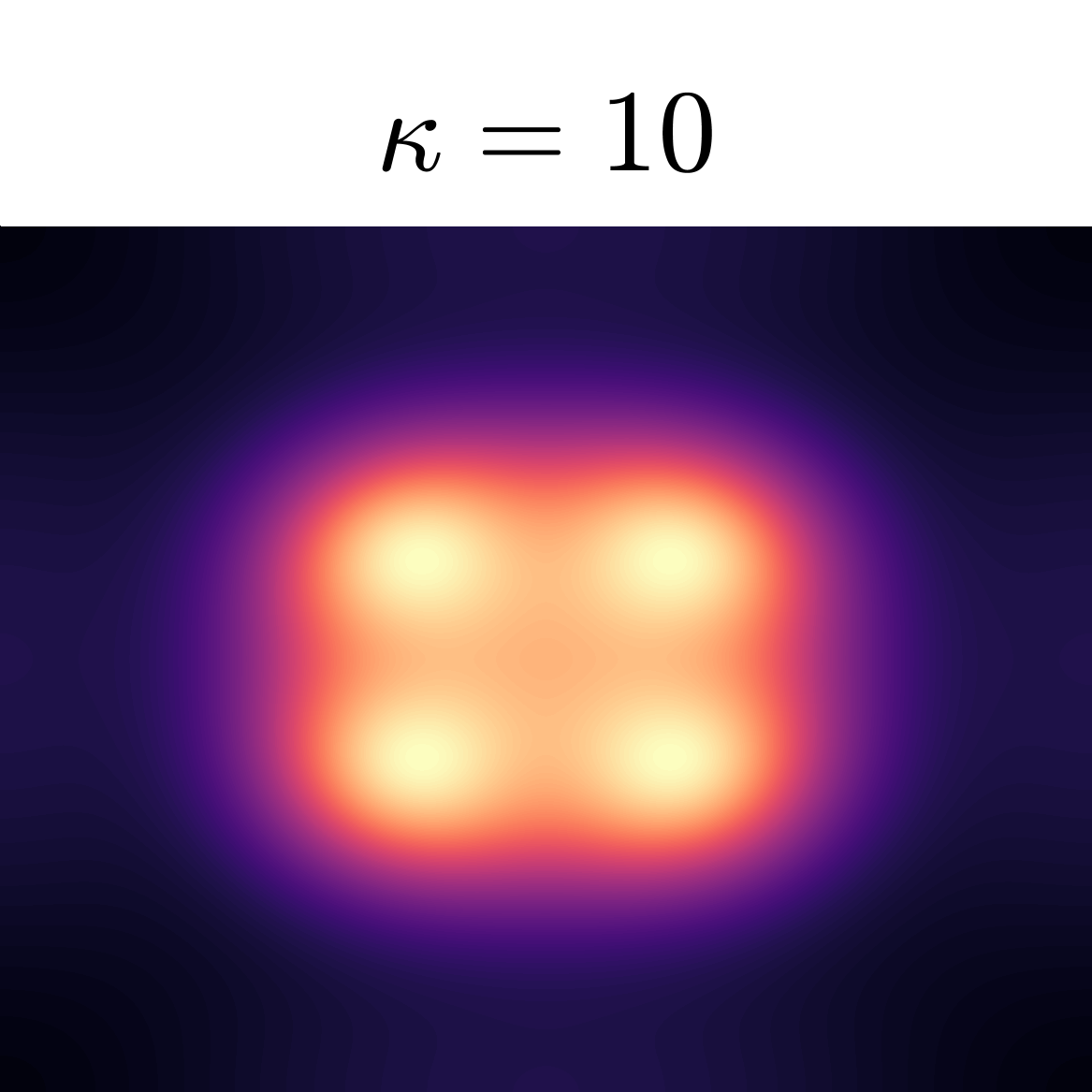}
\end{minipage}
\begin{minipage}{0.19\textwidth}
\includegraphics[scale=0.14]{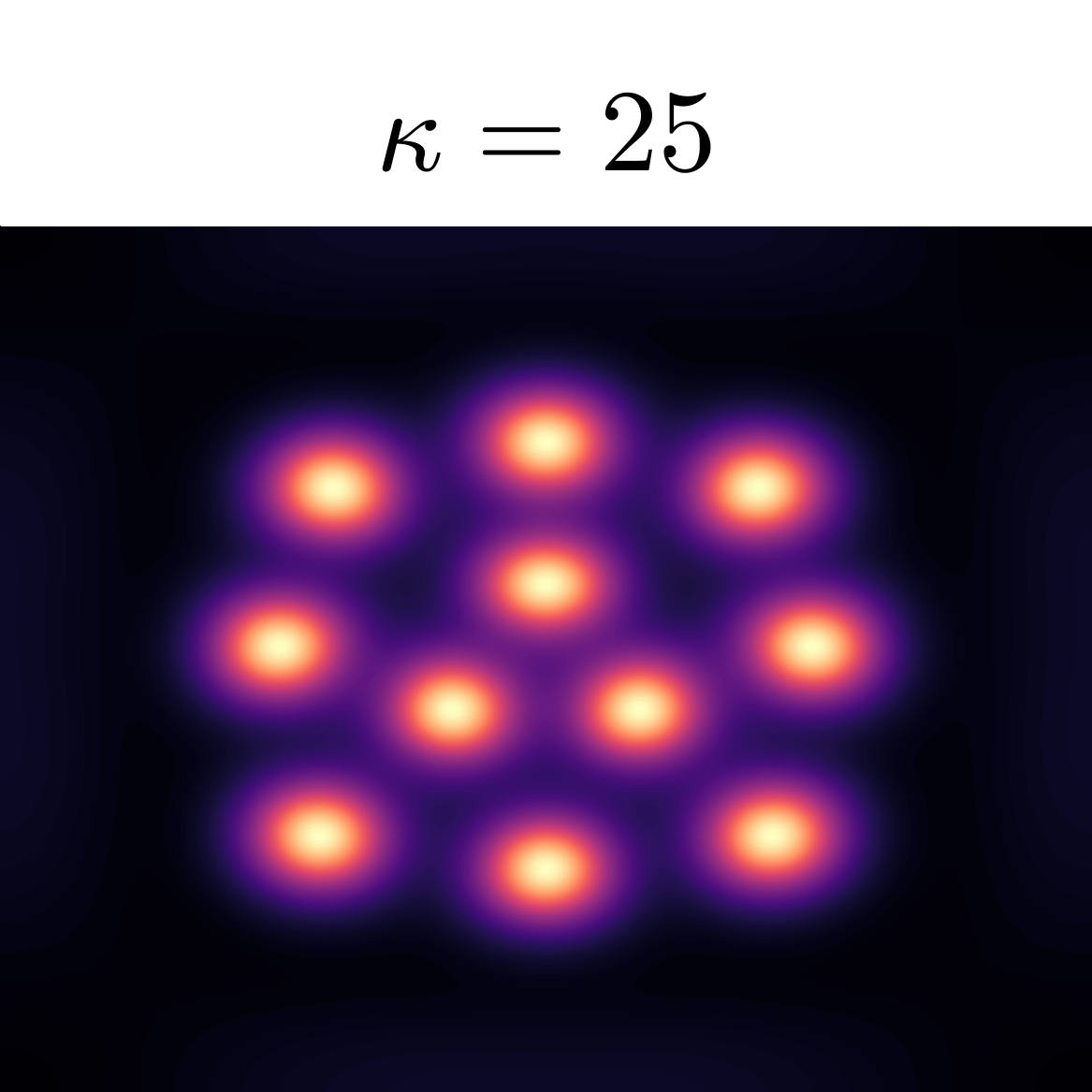}
\end{minipage}
\begin{minipage}{0.19\textwidth}
\includegraphics[scale=0.14]{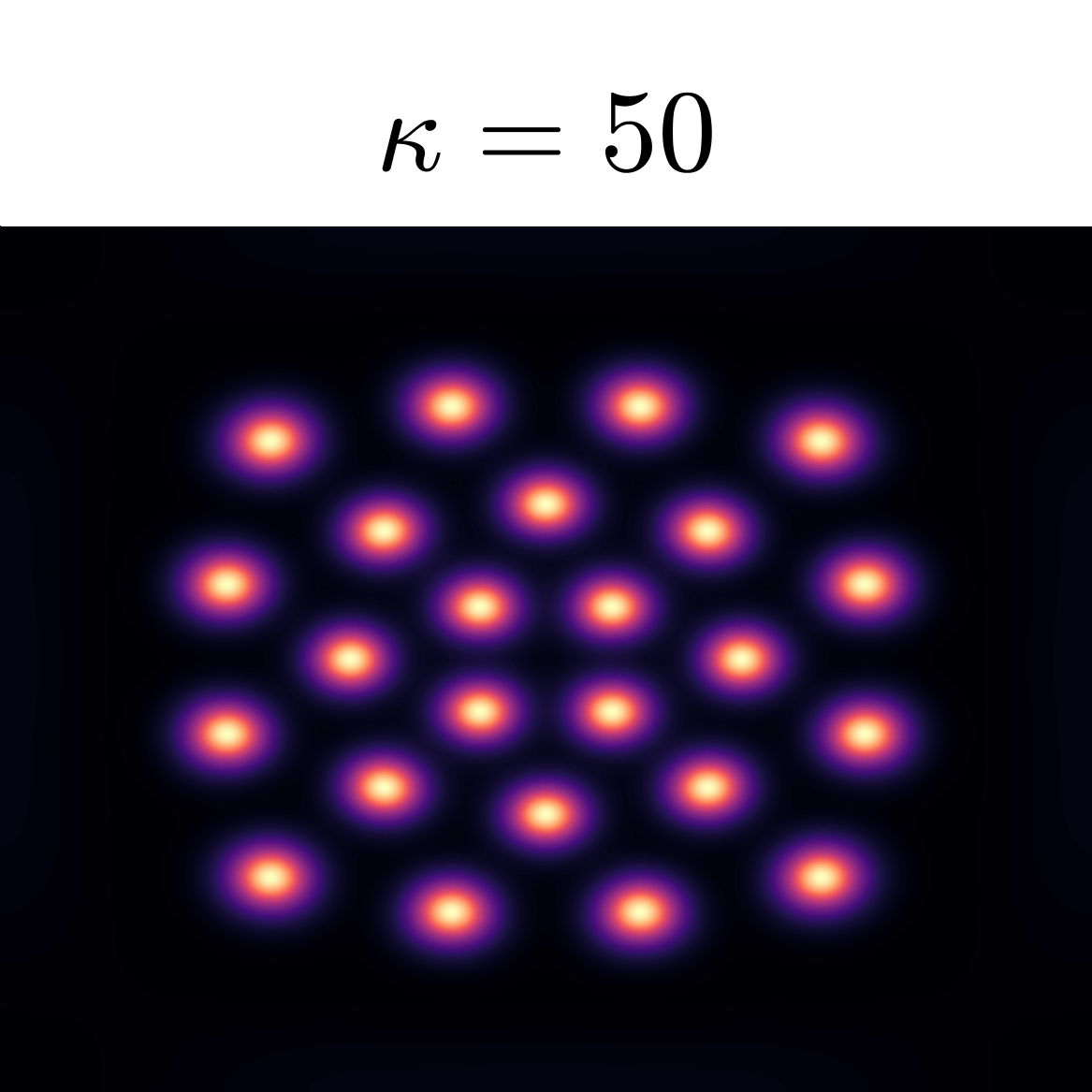}
\end{minipage}
\begin{minipage}{0.19\textwidth}
\includegraphics[scale=0.14]{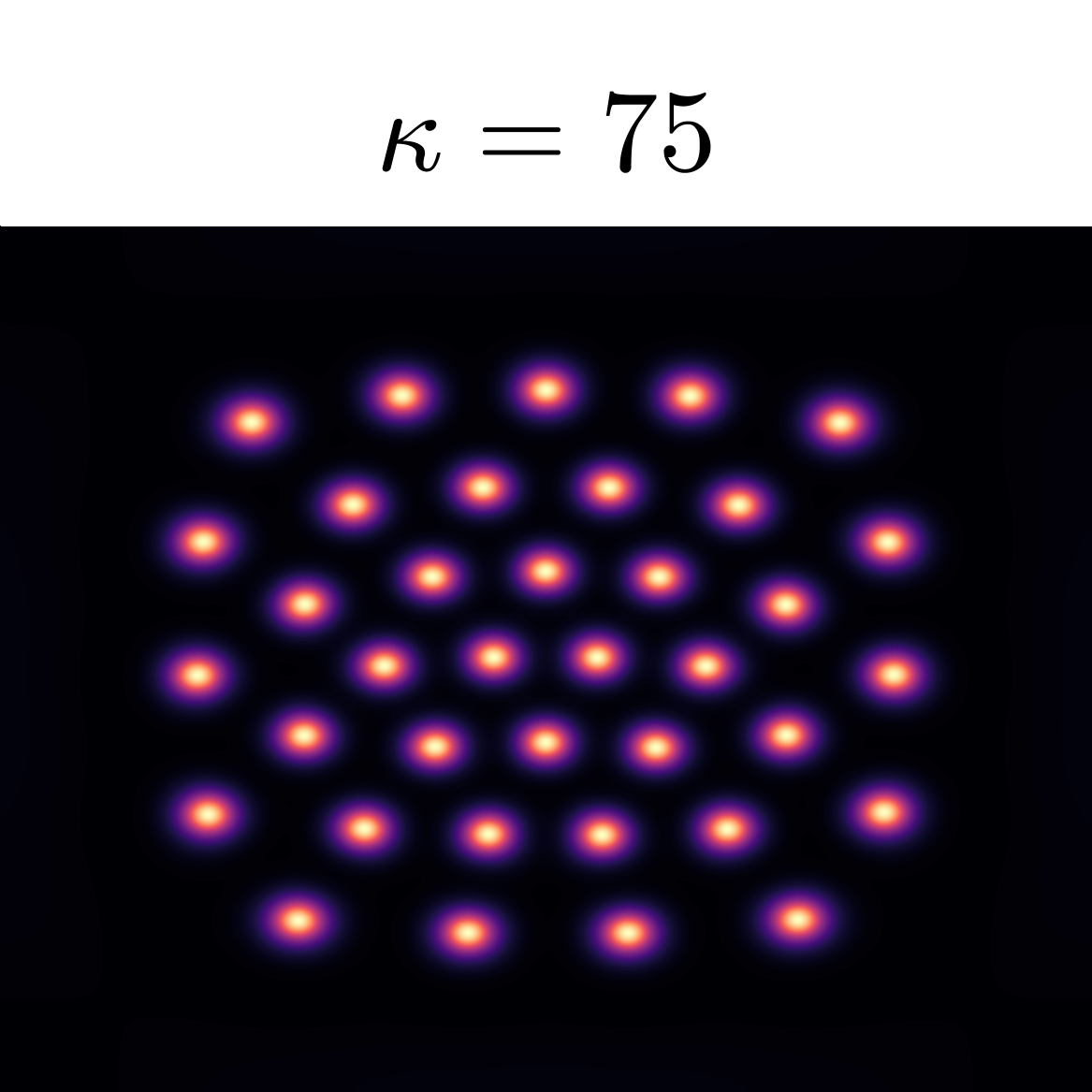}
\end{minipage}
\begin{minipage}{0.19\textwidth}
\includegraphics[scale=0.14]{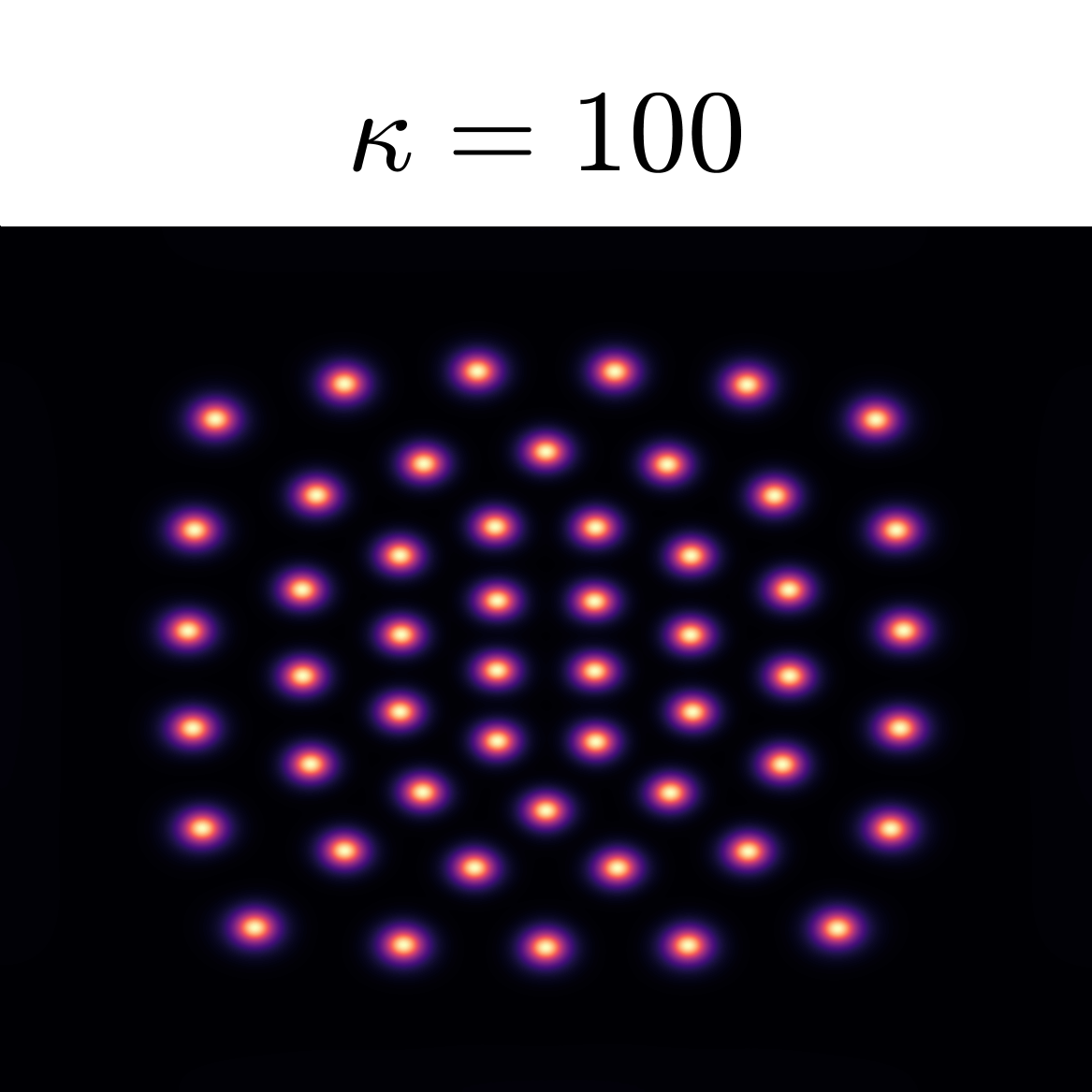}
\end{minipage}
	\caption{In a fixed configuration, vortex patterns in the minimizing density $|u|^2$ for $\kappa=10, 25, 50, 75, 100$.}
	\label{fig:vortices_reference}
\end{figure}

These intricate vortex patterns have motivated the development of a wide range of numerical methods for GL models, such as finite differences \cite{Du05,GJX19}, finite elements \cite{DGP92,DuGP93,GaoS18,HongMaXuChen23,Li17,LiZ15,LiZ17,MaQiao23}, finite volume methods \cite{DuNicolaidesWu98,QuJu05}, and, very recently, vortex-tracking strategies based on ODE systems for the vortex cores \cite{CorsoKemlinMelcherStamm2025}. All these methods address different computational challenges in various settings, both stationary and time-dependent.

In this chapter, we investigate the GL model, in a simplified version where the vector potential $\bfA$ is given, from a numerical perspective using a multiscale approach known as the Localized Orthogonal Decomposition (LOD). Comprehensive introductions to the LOD framework can be found in the monograph~\cite{MaPeBook21} and the survey~\cite{AHP21Acta}, which also highlight its broad range of applications. These include, for example, wave propagation \cite{AbdHe17,FHP24,MaPe19,MaVer22wave,KrMa25}, eigenvalue problems \cite{HePer23,KolVer25,MaPe15,Map17}, and nonlinear PDEs \cite{DHW24,HeWa22,MaVer22,Ver22}.
The LOD, introduced in \cite{MP14} for elliptic multiscale problems and first proposed for the GL model in \cite{DH24}, was rigorously analyzed and applied in \cite{BDH25} and turned out to be a beneficial approximation approach for GL energy minimizers. In particular, it was shown that the resolution condition on the numerical mesh in dependence on the GL parameter $\kappa$ can be significantly relaxed compared to the resolution condition in classical finite element spaces as shown in \cite{DH24}. Recently, the pollution effect was further investigated for classical finite elements of arbitrary order \cite{CFH25} and for low-regularity regimes on non-convex domains \cite{Doerich2025}. By developing a new error analysis in \cite{CFH25}, the authors extended previous results from \cite{DH24} and succeeded in tracing all $\kappa$-dependencies in the resolution condition, thereby closing the remaining gaps in the dependency analysis. Using the new approach for the error analysis, we are able to refine the error analysis from \cite{BDH25} in the context of LOD multiscale methods.\\

The rest of this chapter is organized as follows. In Section \ref{sec:GL} we introduce the analytical framework for minimizers of the free GL energy and collect some preliminary results including the discussion of first and second order optimality conditions. In Section \ref{sec:error_analysis} we introduce the multiscale approximation spaces based on LOD and consider minimizers in this discrete spaces. We state our main results giving error estimates of discrete minimizers under weak resolution condition and derive the error analysis in the remaining section. Finally, Section \ref{sec:numerics} is devoted to numerical computation of (discrete) minimizers. We introduce a nonlinear conjugate Sobolev gradient descent method to compute GL minimizers in numerical simulations and close this chapter with numerical examples.

\section{The Ginzburg-Landau energy minimization problem} 
\label{sec:GL}

Before diving in the mathematical framework we fix our notation.

\medskip
For $2 \le p \le \infty$, let $L^p(\Omega)$ denote the standard Lebesgue space of complex-valued functions on $\Omega$, equipped with the norm $\|\cdot\|_{L^p}$. In the case $p=2$, we endow $L^2(\Omega)$ with the \emph{real} inner product
\begin{align*}
(v,w)_{L^2} := \Re \int_{\Omega} v \, \overline{w} \, \mathrm{d}x, 
\qquad v,w \in L^2(\Omega),
\end{align*}
where $\overline{w}$ denotes the complex conjugate of $w$. With this convention, $L^2(\Omega)$ becomes a real Hilbert space of complex-valued functions. 
The same notation is used for vector fields, in which case the Euclidean dot product is applied inside the integral.

For any integer $k \ge 1$, we denote by $H^{k}(\Omega)$ the standard Sobolev space of complex-valued functions in $L^2(\Omega)$ whose weak derivatives up to order $k$ also belong to $L^2(\Omega)$. The corresponding norm is defined by
\begin{align*}
\|u\|_{H^{k}} := \Biggl( \sum_{|\alpha| \le k} \|\partial^{\alpha} u\|_{L^2}^2 \Biggr)^{1/2},
\end{align*}
where $\alpha$ is a multi-index. We equip $H^1(\Omega)$ with the real inner product
\begin{align*}
(v,w)_{H^1} := \Re \int_{\Omega} v \, \overline{w} + \nabla v \cdot \overline{\nabla w} \, \mathrm{d}x.
\end{align*}

Vector potentials are real-valued vector fields. We use the notation $L^p(\Omega;\R^d)$ and $\boldsymbol{H}^k(\Omega) := H^k(\Omega;\R^d)$ for the corresponding vector-valued spaces. For a function $v \in H^1(\Omega)$ and a functional $f \in (H^1(\Omega))'$ in the dual space we denote by $\langle f, v \rangle$ the dual pairing. Finally, we write $a \lesssim b$ if there exist a constant $C > 0$ which is independent of $\kappa$ and $h$ such that $a \le C b$.
\medskip

As described in the introduction the order parameter $\ucol \in H^1(\Omega)$ and the magnetic potential $\bfA \in \boldsymbol{H}^1(\Omega)$ of the superconductor are given as minimizers of the energy
\begin{align*} 
{\EGL(\ucol,\bfA) = \frac{1}{2} \int_\Omega \Big|\frac{\ci}{\kappa} \nabla \ucol + \bfA \,\ucol \Big|^2 + \frac{1}{2} \left( 1- |\ucol|^2 \right)^2 
+ |\curl \bfA - \boldsymbol{H} |^2
+ |\div \bfA|^2
\,\dx}
\end{align*}
on the space $H^1(\Omega) \times \boldsymbol{H}^1(\Omega)$. Compared to the definition of $\EGL$ in \eqref{eq:energy_functional}, we see that the above energy contains the additional term $\int_\Omega |\div \bfA|^2 \,\dx$. This term imposes a gauge condition to ensure uniqueness of $\bfA$ without changing the relevant physical quantity of interest, that is, the magnetic field \,$\curl \bfA$. 

In order to simplify the presentation of the main theoretical findings, we want to consider a simplified energy functional. The simplification is based on the following general observations which can be made: 
\begin{itemize}
\item  The order parameter $\ucol$ is typically the \quotes{interesting} quantity, exhibiting complex vortex structures.
\item  On convex domains it can be proved (cf. \cite{DDH24}) that it holds $\bfA \in \boldsymbol{H}^2(\Omega)$ and
\begin{align*}
{ \| \bfA \|_{L^{\infty}} \,\, \lesssim \,\, \| \bfA \|_{H^2} \,\, \lesssim \,\, 1 + \| \boldsymbol{H} \|_{L^2} + \| \curl \boldsymbol{H} \|_{L^2}}\qquad
\mbox{{independent} of ${\kappa}$}. 
\end{align*}
\end{itemize}
Hence, in a simplified model we can assume that $\bfA$ is given and smooth. With this, we consider the following (reduced) Ginzburg--Landau model, which only seeks the order parameter $\ucol \in H^1(\Omega)$ as \textit{global minimizer} of the energy
\begin{align*} 
{\Ecol({v}) := \frac{1}{2} \int_\Omega \Big| \frac{\ci}{\kappa} \nabla {v} + \bfA \,{v} \Big|^2 + \frac{1}{2} \left( 1- |{v}|^2 \right)^2\,\dx},
\end{align*}
i.e. $u \in H^1(\Omega)$ fulfills
\begin{align}
\label{minimizer-energy-def} 
E(u) \,\,= \min_{v \in H^1(\Omega)} E(v).
\end{align}
The given quantities are $\Omega$, $\bfA$ and $\kappa$, where we make the following universal assumptions:
\begin{enumerate}[label={(A\arabic*)}]
\item\label{A0} $\Omega \subset \R^d$ is a convex bounded polytope, with $d = 2,3$.
\item\label{A1} The magnetic vector potential $\bfA \in L^{\infty} (\Omega;\mathbb{R}^d)$ fulfills 
\begin{align*}
	{\mbox{{div}}\,\bfA = 0} \,\,\text{ in } {\Omega}
	\qquad
	\mbox{and}
	\qquad 
	{\bfA \cdot \boldsymbol{n} =0}  \,\,\text{ on } {\partial \Omega}.
\end{align*}
\item\label{A2} The material parameter $\kappa > 0$ is constant and real. 
\end{enumerate}
Note that the divergence of $\bfA$ in \ref{A1} is understood in terms of weak derivatives. If $\bfA$ has a weak divergence, then it also has a normal trace $\bfA \cdot \boldsymbol{n}$ on the boundary of a Lipschitz domain $\Omega$. Hence, the condition ${\bfA \cdot \boldsymbol{n} =0}$ is well-defined.

\subsection{Well-posedness and first order conditions}

In this section, we recall classical results concerning the existence of a global minimizer 
$u \in H^1(\Omega)$ of the energy functional $E$ under assumptions \ref{A1}--\ref{A2}. 
We further analyze its relation to the Ginzburg--Landau equation and 
establish regularity and stability estimates for the minimizer in terms of $\kappa$ 
under assumption \ref{A0}. These stability estimates are fundamental for quantifying 
the approximation properties of finite element discretizations. \\

The existence of a global minimizer $\ucol$ of $E$ can be shown by a classical compactness argument and the weak lower-semi continuity of $E$ on $H^1(\Omega)$, cf. \cite{DGP92}. 

\begin{lemma}
Let {$\Omega \subset \R^{d}$} ({$d=2,3$}) be a bounded Lipschitz-domain and let \ref{A1}-\ref{A2} be fulfilled. Then there exists a global minimizer $\ucol \in H^1(\Omega)$ of $E$, i.e. \eqref{minimizer-energy-def} holds.
\end{lemma}

Besides, of global minimizers we also consider \textit{local minimizers} that minimize the energy $E$ at least in a local neighborhood. In our terminology, unless otherwise specified, we name $u$ a minimizer if it is a local minimizer which may or may not be a global minimizer. \\

For a minimizer $u$, classical optimization theory implies that it needs to be a \textit{critical point} of $E$ so that the first Fr\'echet derivative if it exists needs to vanish, i.e., $\langle E'(u), w \rangle = 0$ for all $w \in H^1(\Omega)$. This is usually called the \textit{first order optimality condition}. Indeed, it is a simple matter of fact that the energy functional $E: H^1(\Omega) \rightarrow \R$ is Fr\'echet differentiable on the real Hilbert space $H^1(\Omega)$ and its derivative at a point $v \in H^1(\Omega)$ is given by
\begin{align*}
\langle E^{\prime}(v) , {w} \rangle
\,\,=\,\, \Re \int_\Omega \bigl( \tfrac{\ci}{\kappa}  \nabla v +   \bfA v \bigr)  \cdot \overline{\bigl(\tfrac{\ci}{\kappa}  \nabla {w} +  \bfA {w} \bigr)}  
	+
	 \bigl( |v|^2 -1 \bigr)  v \overline{{w}} 
	\,\,\dx  \quad \mbox{for all } w \in H^1(\Omega).
\end{align*}
Thus, by the first order condition we obtain that any minimizer $u \in H^1(\Omega)$ satisfies the \textit{Ginzburg-Landau equation} (GLE)
\begin{align} \label{eq:GL_variational}
	\bigl(  \tfrac{\ci}{\kappa}  \nabla u +   \bfA u, \tfrac{\ci}{\kappa}  \nabla w +   \bfA w \bigr) + \bigl((|u|^2 - 1)u,w \bigr) = 0 \quad \mbox{for all } w \in H^1(\Omega).
\end{align}
Under the assumptions \ref{A0}-\ref{A2}, this Neumann-problem is $H^2$-regular, cf. \cite{Gr11}, so that $u \in H^2(\Omega)$ solves the GLE in its strong form
\begin{align} \label{eq:GL_strong}
	\bigl( \tfrac{\ci}{\kappa} \nabla + \bfA \bigr)^2 u + \bigl( |u|^2 - 1 \bigr) u = 0
\end{align}
with the natural boundary condition $\bfA \cdot \boldsymbol{n} =0$ on $\partial \Omega$. \\

As we state now any critical point, and therefore any minimizer, satisfies certain $\kappa$-dependent stability estimates that will play a crucial role in our subsequent analysis. Its proof can be found in \cite{CFH25} for arbitrary critical points which extends the previous result from \cite{DH24} which is limited to global minimizers. 

\begin{lemma}{\cite[Thm.~2.5 \& Thm.~2.6]{CFH25}} \label{lem:existence}
Let {$\Omega \subset \R^{d}$} ({$d=2,3$}) be a bounded Lipschitz-domain, let \ref{A1}-\ref{A2} be fulfilled, and let $u \in H^1(\Omega)$ be a critical point of $E$, i.e., $\langle E'(u), v \rangle = 0 $ for all $v \in H^1(\Omega)$. Then the following stability bounds hold:
\begin{align*}
	| u(x)| \le 1 \quad \text{a.e.~in } \Omega, \quad \| u \|_{L^2} \lesssim 1, \quad \| \nabla u \|_{L^2} \lesssim \kappa \| u\|_{L^2}. 
\end{align*}
If in addition \ref{A0} is satisfied, then $u \in H^2(\Omega)$ with the bounds
\begin{align*}
	\| \nabla u \|_{L^4} \lesssim \kappa \quad \text{and} \quad \| D^2 u \|_{L^2} \lesssim \kappa^2
\end{align*}
where $D^2 u$ denotes the Hessian of $u$.
\end{lemma}

The pointwise bound in Lemma \ref{lem:existence} shows that the density of any minimizer $|\ucol|^2$ indeed takes values between zero and one which is in align with its physical interpretation of the density of superconducting electrons - the Cooper pairs. \\
Due to the $\kappa$-dependent stability estimates it is natural to introduce the $\kappa$-weighted norms
\begin{align*}
	\| v \|_{H^1_\kappa} := \Big( \tfrac{1}{\kappa^2} \| \nabla v \|_{L^2}^2 + \| v \|_{L^2}^2 \Big)^{1/2}, \quad 	\| v \|_{H^2_\kappa} := \Big( \tfrac{1}{\kappa^4} \| D^2 v \|_{L^2}^2 + \tfrac{1}{\kappa^2} \| \nabla v \|_{L^2}^2 + \| v \|_{L^2}^2 \Big)^{1/2}
\end{align*}
so that we have for any critical point $\ucol \in H^2(\Omega)$ the bounds $\| u \|_{H^1_\kappa} \lesssim 1$ and $\| u \|_{H^2_\kappa} \lesssim 1$.

\subsection{Gauge invariance and second order conditions}

If $u \in H^1(\Omega)$ is a critical point of $E$, a sufficient condition for $u$ to be an isolated minimizer is that the second Fr\'echet derivative $E''(u): H^1(\Omega \times H^1(\Omega) \rightarrow \R$ is positive definite. It is given by
{\begin{align*}
	\langle E^{\prime\prime}(\ucol) {v} , {w} \rangle = \Re   \int_\Omega \bigl( \tfrac{\ci}{\kappa} {\nabla v} + \bfA {v} \bigr)  \cdot \overline{\bigl( \tfrac{\ci}{\kappa} \nabla {w} + \bfA {w} \bigr)} + ( |\ucol|^2 -1  ) {v} \overline{{w}} + 2 \,\Re\!(\ucol \overline{{v}}) \, \ucol \overline{{w}} \dx,
\end{align*}
but, as we shall see, this form cannot be positive definite. A simple observation is that for any $\theta \in [0,2\pi)$ we have
\begin{align*}
	E(u) = E(e^{i\theta} u).
\end{align*}
This property, known as gauge invariance of the energy functional, implies that if $u \in H^1(\Omega)$ is a minimizer of $E$, then so is $e^{i\theta} u$. In fact, the entire orbit $\{ \gamma(\theta) \}_{\theta \in [0,2\pi)}$, where
\begin{align*}
	\gamma: [0,2\pi) \rightarrow H^1(\Omega), \quad \theta \mapsto e^{i\theta} u,
\end{align*}
forms a continuum of minimizers. From the first order optimality condition it follows that
\begin{align*}
	0 = \frac{\mathrm{d}^2}{ \mathrm{d} \theta^2} E(\gamma(\theta))_{| \theta = 0} = \langle E''(\gamma(0)) \gamma'(0), \gamma'(0) \rangle + \langle E'(\gamma(0)), \gamma''(0) \rangle = \langle E''(u) \ci u, \ci u \rangle
\end{align*}
which shows that $E''(u)$ is not positive definite. In particular, $\ci u$ is an eigenfunction of $E''(u)$ to a zero eigenvalue. However, a natural assumption is that this is the only singular direction so that $E''(u)$ is positive definite on the orthogonal complement $(\ci u)^\bot$. We call this property quasi-isolation of a minimizer.

\begin{definition}[Quasi-isolated minimizer]
A function $u \in H^1(\Omega)$ is called a \textit{quasi-isolated} (local) minimizer if it satisfies
\begin{align*}
	\langle E'(u), v \rangle = 0 \quad \text{for all } v \in H^1(\Omega)
\end{align*}
and
\begin{align*}
		\langle E''(u) v, v \rangle > 0 \quad \text{for all } v \in (\ci u)^\bot \setminus \{ 0 \},
\end{align*}
where $\orthiu := \{ v \in H^1(\Omega): \, (\ci u, v) = 0 \}$ is the $L^2$-orthogonal complement of $\ci u$ in $H^1(\Omega)$. Recall here that orthogonality is only with respect to the real part of the integral. If in addition \eqref{minimizer-energy-def} holds, then $u$ is called a \textit{quasi-isolated global minimizer}.
\end{definition}

Unfortunately, it remains an open problem whether $\ci u$ is the only singular direction of $E''(u)$ assuming that $u$ is a local or global minimizer, and we are not aware of any rigorous justification. Nevertheless, once a local minimizer has been computed numerically, one can at least verify if it is quasi-isolated by evaluating the spectrum of $E''(u)$, as it was done in \cite{BDH25}. In this case, quasi-isolation requires that the spectrum of $E''(u)$ contains a simple zero eigenvalue, while all other eigenvalues are strictly positive. \\

We will conclude this section by presenting some preliminary results regarding the bilinear form induced by the second Fr\'echet derivative, $E''(u)$, which we will use in our analysis later on. \\
 
A direct consequence of the quasi-isolation of a minimizer is that the bilinear form $\langle E''(u) \cdot, \cdot \rangle$ is coercive on $(\ci u)^\bot$.

\begin{lemma}{\cite[Prop. 2.2]{BDH25}} \label{lem:coercivity}
Let {$\Omega \subset \R^{d}$} ({$d=2,3$}) be a bounded Lipschitz-domain, \ref{A1}-\ref{A2} be fulfilled, and let $u \in H^1(\Omega)$ be a quasi-isolated minimizer of $E$. Then there exists a constant $\rho = \rho(\kappa) > 0$ such that
\begin{align*}
	\langle E''(u) v ,v \rangle \ge \rho(\kappa)^{-1} \| v \|_{H^1_\kappa}^2 \quad \text{for all } v \in (\ci u)^\bot \setminus \{ 0 \}.
\end{align*}
Furthermore, we have
\begin{align*}
	\langle E''(u) v ,w \rangle \lesssim \| v \|_{H^1_\kappa} \| w \|_{H^1_\kappa} \quad \text{for all } v, w \in H^1_\kappa(\Omega).
\end{align*}
Note that the continuity of $E^{\prime\prime}(u)$ also implies $\Csol \gtrsim 1$.
\end{lemma}

\begin{proof}
The coercivity is proved by applying the Courant-Fischer theorem and using a G{\aa}rding-inequality for $\langle E''(u) \cdot,\cdot \rangle$ on $H^1(\Omega)$, see \cite[Prop. 2.2]{BDH25}. The proof of the continuity is straightforward using the pointwise bounds from Lemma \ref{lem:existence} and \ref{A1}. 
\end{proof}
As in the case of elliptic operators on standard Sobolev spaces, the operator $E''(u)$ exhibits higher regularity on $\orthiu$ when the domain is convex.
\begin{lemma}[Regularity of solutions to $E^{\prime\prime}(u)z = {f}$ in $\orthiu$]
\label{senv-theorem-reg-secE-pde}
Assume the setting of Lemma~\ref{lem:coercivity}, and suppose that {$\Omega \subset \R^{d}$} is convex. Then, for every {${f} \in L^2(\Omega)$}, there exists a unique solution {$z \in \orthiu \cap H^2(\Omega)$} of
\begin{align*}
\langle E^{\prime\prime}(u) z , v \rangle
= (f, v)
\qquad
\text{for all } v \in \orthiu.
\end{align*}
Moreover, the following estimates hold:
\begin{eqnarray*}
\| z \|_{H^1_{\kappa}}
&\lesssim& \Csol \,
\| f \|_{L^2}
\qquad\mbox{and}\qquad
\tfrac{1}{\kappa^2} | z |_{H^2(\Omega)}
\,\,\,\lesssim\,\,\,
 \Csol \,
\| f \|_{L^2},
\end{eqnarray*}
where $\Csolinv$ denotes the coercivity constant from Lemma~\ref{lem:coercivity}. Finally, $z \in H^2(\Omega)$ fulfills  
\begin{align}
\label{alt-char-z}
(\tfrac{\ci}{\kappa} \nabla z + \bfA z , \tfrac{\ci}{\kappa} \nabla v + \bfA v ) = ( g ,v) \qquad \mbox{for all } v\in H^1(\Omega), 
\end{align}
where 
$g := f + (1 \!-\!  |u|^2  ) z - 2\, \Re  (u \overline{z}) \, u -  \| u\|_{L^2}^{-2} (  f,   \ci u  )  \, \ci u$. 
\end{lemma}
\begin{proof}
Existence of $z\in \orthiu$ and the stability bounds are a direct consequence of Lemma \ref{lem:coercivity}.
The proof of the regularity statement is given in \cite[Corollary 2.9]{DH24} and relies on standard elliptic regularity theory in combination with a decomposition of $z$ according to the orthogonal splitting $H^1(\Omega) = \mbox{span} \{ \ci u \} \oplus \orthiu$. The splitting also allows to write $z $ in the characterization \eqref{alt-char-z}. For this, we proceed as in \cite{DH24} and note that by the definition of $E^{\prime\prime}(u)$, we have
\begin{align*}
	( \tfrac{\ci}{\kappa} {\nabla z} + \bfA {z} , \tfrac{\ci}{\kappa} \nabla v^{\perp} + \bfA v^{\perp} ) \,=\, (\,f + (1 \!-\!  |u|^2  ) z - 2\, \Re  (u \overline{z}) \, u  , v^{\perp} ) =: (f_{u,z} , v^{\perp}  )
\end{align*}
for all $v^{\perp} \in \orthiu$. For arbitrary $v \in H^1(\Omega)$ we write $v = v^{\perp}+ \alpha_v (\ci u)$, where $v^{\perp} \in \orthiu$ is the $L^2$-projection of $v$ onto $\orthiu$ and $\alpha_v := (v,\ci u) \| u \|_{L^2}^{-2}$. 

Consequently, we obtain
\begin{eqnarray*}
( \tfrac{\ci}{\kappa} {\nabla z} + \bfA {z} , \tfrac{\ci}{\kappa} \nabla v + \bfA v )
&=& 
\alpha_v \, ( \tfrac{\ci}{\kappa} {\nabla z} + \bfA {z} , \tfrac{\ci}{\kappa} \nabla (\ci u ) + \bfA (\ci u) )
+ (f_{u,z} , v^{\perp}  ) \\
&\overset{E^{\prime}(\ci u) = 0}{=}& \alpha_v ( (1- |u|^2) \ci u , z ) + (f_{u,z} , v^{\perp}  ) \\
&=&  \alpha_v ( (1- |u|^2) \ci u , z ) + (f_{u,z} , v -\alpha_v \ci u  ).  
\end{eqnarray*}
The statement follows with $\alpha_v = (v,\ci u) \| u \|_{L^2}^{-2}$ and
exploitng $(  2\, \Re  (u \overline{z}) \, u,   \ci u  )=0$.
\end{proof}

Despite the general lack of coercivity, the operator $E^{\prime\prime}(u)$ always fulfills a G{\aa}rding inequality.

\begin{lemma}[G{\aa}rding inequality for $\langle E^{\prime\prime}(u) \,\cdot\,,\cdot \rangle$] 
\label{senv-lemma-gaarding}
Let $\Omega \subset \R^{d}$, $d \in \{2,3\}$, be a bounded Lipschitz domain.  
Assume that conditions \ref{A1}--\ref{A2} hold, and let $w \in H^1(\Omega)$ be arbitrary.  
Then the following G{\aa}rding-type inequality is satisfied:
\begin{align*}
\langle E^{\prime\prime}(w)v, v \rangle 
\,\,\ge\,\, \tfrac{1}{2}\, \|v\|_{H^1_{\kappa}}^2 
- c_{\bfA}\, \|v\|_{L^2}^2,
\qquad \text{for all } v \in H^1(\Omega),
\end{align*}
where $c_{\bfA} = \tfrac{3}{2} + \|\bfA\|_{L^{\infty}(\Omega)}^2$.
\end{lemma}

\begin{proof}
A proof is given in \cite[Lemma 4.1]{CFH25}. It is based on expanding the term $|\tfrac{\ci}{\kappa}\nabla v + \mathbf{A}v|^2$, applying Young's inequality to control the mixed term $ \nabla v \cdot \mathbf{A}v$, and then using $ |a|^2 + |b|^2 \ge \tfrac{1}{2}(|a| + |b|)^2$ to derive the desired G{\aa}rding-type lower bound for $\langle E''(\ucol)v, v\rangle$.
\end{proof}

\section{Minimizers in multiscale approximation spaces} 
\label{sec:error_analysis}

We now turn to the numerical approximation of minimizers of the Ginzburg--Landau energy in suitable multiscale spaces and establish the corresponding computational framework that was first introduced in \cite{BDH25} for the reduced GLE and later extended in \cite{DDH24} to the full model. 

As outlined in the introduction, we construct discrete approximation spaces using the multiscale method known as the Localized Orthogonal Decomposition (LOD). The core idea of the LOD is to design basis functions for the approximation space that are specifically adapted to the problem at hand. In our setting, the structure of the Ginzburg–Landau model enters through the bilinear form
\begin{align} \label{eq:abeta}
	a(v,w) := (\tfrac{\ci}{\kappa} \nabla v + \bfA v, \tfrac{\ci}{\kappa} \nabla w + \bfA w), \quad v,w \in H^1(\Omega).
\end{align}
In particular the bilinear form $a(\cdot,\cdot)$ consists of the kinetic part of the Ginzburg-Landau equation \eqref{eq:GL_variational}, induced by the so-called magnetic Neumann Laplacian $(\tfrac{\ci}{\kappa} \nabla + \mathbf{A})^2$.
Starting from the classical $H^1$-conforming $\mathcal{P}_1$-Lagrange finite element space
\begin{align*}
	V_h = \{ v_h \in H^1(\Omega): \, v_{h} |_T \in \mathcal{P}_1(T) \text{ for all } T \in \mathcal{T}_h \},
\end{align*}
where $\mathcal{T}_h$ denotes a quasi-uniform, conforming, shape regular mesh of $\Omega$, we define the $L^2$-projection
\begin{align*}
	\Ph: H^1(\Omega) \rightarrow V_h ,\quad (\Ph v -v, w_h) = 0 \quad \text{for all } w_h \in V_h.
\end{align*}
Its kernel on $H^1(\Omega)$
\begin{align*}
	W := \ker (\Ph)
\end{align*}
contains, roughly speaking, the fine-scale information that the finite element space 
$V_h$ cannot resolve, and is therefore often referred to as the \textit{detail space}. Although the bilinear form $a(\cdot,\cdot)$ does not define an inner product on the entire solution space $H^1(\Omega)$, due to its lack of coercivity, the following lemma from \cite{BDH25} establishes that it is coercive when restricted to the detail space $W$.

\begin{lemma} \label{lem:abeta_coercivity}
Let \ref{A0}-\ref{A2} be fulfilled. Then $a(\cdot,\cdot)$ is uniformly $H^1_{\kappa}$-continuous on $H^1(\Omega)$ with
\begin{align*}
	a(v,w) \,\,\,\lesssim\,\,\, \| v \|_{H^1_\kappa} \| w \|_{H^1_\kappa} \qquad \mbox{for all } v,w \in H^1(\Omega),
\end{align*}
and there exists a constant $C > 0$ that is independent of $\kappa$ and $h$ such that if $h \kappa \le C$, it holds
\begin{align*}
	a(v,v) \,\,\,\ge\,\,\, \frac{1}{4} \| v \|_{H^1_\kappa}^2 \qquad \text{for all } v \in W= \ker(\Ph).
\end{align*} 
\end{lemma}

For the proof of Lemma \ref{lem:abeta_coercivity} we refer to \cite{BDH25}. Since $a(\cdot,\cdot)$ is coercive on the detail space $W$ we can define the corrector
\begin{align}
\label{corrctor-op-def}
	\Cor: H^1(\Omega) \rightarrow W, \quad a(\Cor v, w) = a(v,w) \quad \text{for all } w \in W.
\end{align}
In particular, due to Lemma \ref{lem:abeta_coercivity}, this corrector is an $H^1_\kappa$-stable projection onto $W$ since we have
\begin{align*}
 \| \Cor v \|_{H^1_\kappa}^2 \lesssim a( \Cor v , \Cor v ) =  a( v ,\Cor v ) \lesssim 
 \| \Cor v \|_{H^1_\kappa} \| v \|_{H^1_\kappa}, \quad \mbox{and thus,}
 \quad  \| \Cor v \|_{H^1_\kappa} \lesssim  \| v \|_{H^1_\kappa}.
\end{align*}
The desired multiscale approximation space for GL minimizers is now defined in the spirit of the LOD by
\begin{align} \label{eq:LODspace}
	\Vlod := (\mathrm{Id} - \Cor) V_h.
\end{align}

In this new approximation space we seek for minimizers of the GL energy: We call a function $\ulod \in \Vlod$ a \textit{discrete global minimizer} if it satisfies
\begin{align*}
	E(\ulod) = \min_{v \in \Vlod} E(v).
\end{align*}
Clearly, since $\Vlod \subset H^1(\Omega)$, we have that such a discrete global minimizer exists and $E(u) \le E(\ulod)$ holds. As in the continuous setting, we can also consider \textit{discrete local minimizers} which minimize the energy in a local neighborhood and we name $u$ a discrete minimizer if it is at least a discrete local minimizer. If $\ulod \in \Vlod$ is a discrete minimizer it satisfies the first order condition meaning that $\ulod$ is a critical point of $E$ in $\Vlod$, i.e,
\begin{align*}
	\langle E'(\ulod), v\rangle = 0 \quad \text{for all } v \in \Vlod.
\end{align*}
Furthermore, one shows in analogy to the continuous setting that $\| \ulod \|_{H^1_\kappa} \lesssim 1$. However, also discrete minimizers cannot be unique due to the symmetry under complex phase shifts, $E(\ulod) = E(e^{\ci \theta} \ulod)$ for all $\theta \in [0,2\pi)$. We need to take special care of this when comparing discrete minimizers $\ulod$ to exact minimizers $u$ on the entire space $H^1(\Omega)$ in the subsequent analysis.

\subsection{Main results}

We investigate the approximation error between exact minimizers $u \in H^1(\Omega)$ and discrete minimizers $\ulod \in \Vlod$. In particular, we estimate the errors
\begin{align*}
	E(\ulod) - E(\ucol), \quad  \| \ulod - \ucol\|_{H^1_\kappa}, \quad \| \ulod - \ucol \|_{L^2}
\end{align*}
under suitable resolution conditions on the mesh size $h$ in dependence on the GL parameter $\kappa$. \\

Our first main result concerns the error in energy for global minimizers under the weak resolution conditions $h \lesssim \kappa^{-1}$ and extends the energy estimate from \cite{BDH25}. The proof can be found in our error analysis, cf. Section \ref{sec:proof_energy_estimates}.

\begin{theorem}[Energy error estimate for global minimizers] \label{thm:energy_error_estimate}
Let \ref{A0}-\ref{A2} be fulfilled and let $\Vlod$ with $h \lesssim \kappa^{-1}$ be a family of LOD spaces. Then is holds
\begin{align*}
	0 \,\,\,\le \,\,\, \min_{\vlod \in \Vlod}E(\vlod) - \min_{v \in H^1(\Omega)}E(v) \,\,\,\lesssim \,\,\,\kappa^6 h^6.
\end{align*}
\end{theorem}

Thus the minimum discrete energy level $E(\ulod)$ converges to the exact minimum energy level $E(u)$. However, it remains unclear whether the sequence of discrete global minimizers converges to an exact global minimizer, or only converges up to subsequences. We will see that the convergence to an exact minimizer in $H^1_\kappa(\Omega)$ and $L^2(\Omega)$ requires stronger resolution conditions on $h$. This is the content of our second main result, which addresses the approximation of general global and local minimizers within the discrete multiscale space w.r.t. $H^1_\kappa(\Omega)$ and $L^2(\Omega)$. Again, the result extends the corresponding estimates obtained previously in \cite{BDH25}.

\begin{theorem}
\label{theorem:main-result}
Assume \ref{A0}--\ref{A2} and let $u \in H^1(\Omega)$ be quasi-isolated minimizer of $E$. If the following resolution conditions are fulfilled, 
$$
\rho(\kappa)\,(h\kappa)^2 \,\lesssim \, 1
\qquad\text{and}\qquad
\kappa^{d/2}\rho(\kappa)\,(h\kappa)^3 \,\lesssim \, 1,
$$
then there exists a discrete minimizer $\ulod \in \Vlod$ of $E$ with 
\begin{eqnarray*}
\langle E^{\prime}(\ulod) , v_h \rangle &=& 0 \hspace{93pt} \mbox{ for all } v_h \in \Vlod
\qquad
\mbox{and} \\
\langle E^{\prime\prime}(\ulod) v_h , v_h \rangle &\gtrsim& \rho(\kappa)^{-1} \| v_h \|_{H^1_{\kappa}}^2 \hspace{20pt} \mbox{ for all } v_h \in \Vlod\cap \orthiu
\end{eqnarray*}
such that
\begin{eqnarray*}
\| u - \ulod \|_{H^1_{\kappa}} &\lesssim&   (h\kappa)^3
\end{eqnarray*}
and 
\begin{align}
\label{final-L2-estimate}
\| u - \ulod \|_{L^2} \,\,\,\, \lesssim \,\,  \, (h\kappa)^4 \,+\, \rho(\kappa)\,(h\kappa)^5  +  \rho(\kappa) \,\kappa^{d/2} \,(h\kappa)^6. 
\end{align}
Note that the $L^2$-error estimate becomes asymptotically optimal, i.e., we have $\| u - \ulod \|_{L^2} \lesssim (h\kappa)^4$ for all sufficiently small $h$.
\end{theorem}

We give the proof of our second main result in Section \ref{sec:proof_main_result}.

For comparison, let us quickly compare the approximation properties to this in the standard Lagrange space $V_h$ of the same dimension as $\Vlod$. In this case, the following result can be proved, cf. \cite[Theorem 3.2]{CFH25}.

\begin{theorem}
\label{theorem:main-result-std-FEM}
Assume \ref{A0}--\ref{A2} and let $u \in H^1(\Omega)$ be quasi-isolated minimizer of $E$. If
$$
\kappa^{d/2}\, \rho(\kappa)\,(h\kappa) \,\lesssim \, 1
$$
then there exists a discrete minimizer $u_h \in V_h$ of $E$ with 
\begin{eqnarray*}
\langle E^{\prime}(u_h) , v_h \rangle &=& 0 \hspace{93pt} \mbox{ for all } v_h \in V_h
\qquad
\mbox{and} \\
\langle E^{\prime\prime}(u_h) v_h , v_h \rangle &\gtrsim& \rho(\kappa)^{-1} \| v_h \|_{H^1_{\kappa}}^2 \hspace{20pt} \mbox{ for all } v_h \in V_h\cap \orthiu
\end{eqnarray*}
such that
\begin{eqnarray*}
\| u - u_h \|_{H^1_{\kappa}} &\lesssim&  h\kappa
\end{eqnarray*}
and 
\begin{align*}
\| u - \ulod \|_{L^2} \,\,\,\, \lesssim \,\, \kappa^{d/2} \, \rho(\kappa) \, (h\kappa)^2. 
\end{align*}
\end{theorem}
Comparing Theorem \ref{theorem:main-result} with Theorem \ref{theorem:main-result-std-FEM} reveals that the approximation properties in $V_h$ are substantially weaker: the convergence rates are lower, and a much stronger resolution condition is required to capture the physically correct behavior of the minimizers.

\subsection{Proof of the main results}

In this section we prove our main results, Theorem \ref{thm:energy_error_estimate} and Theorem \ref{theorem:main-result}. We begin by analyzing approximations in the LOD space for arbitrary functions that satisfy suitable regularity assumptions under the weak resolution condition $h \lesssim \kappa^{-1}$. We then establish the energy error estimates stated in Theorem \ref{thm:energy_error_estimate}. A key ingredient in the proof of the $H^1_\kappa$- and $L^2$-estimates is the analysis of the bilinear form and the associated Ritz projection induced by the second Fr\'echet derivative $E''(u)$. We derive error bounds showing that this Ritz projection behaves as a best approximation in the LOD space when stronger resolution conditions are imposed. Finally, by applying a fixed-point argument developed in \cite{CFH25}, we ensure the existence of discrete minimizers located near an exact minimizer and conclude the error estimates stated in Theorem \ref{theorem:main-result}.

\subsubsection{Approximation in the LOD space}

We introduce the LOD decomposition operator 
\begin{align} \label{def-LOD-projection}
		\Rlod : H^1(\Omega) \rightarrow \Vlod, \quad \Rlod(v) := \Ph(v) - (\Cor \circ \Ph)(v).
\end{align}
where we recall $\Ph : H^1(\Omega) \rightarrow V_h$ as the $L^2$-projection on $V_h$ and  $\Cor : H^1(\Omega) \rightarrow W$ as the corrector defined in \eqref{corrctor-op-def}. This operator is a projection onto $\Vlod$ as it readily follows for arbitrary $v_h -  \Cor (v_h) \in \Vlod$ (with $v_h\in V_h$) that
\begin{align*}
\Rlod(v_h -  \Cor (v_h)) = \Ph(v_h -  \Cor (v_h)) - (\Cor \circ \Ph)(v_h -  \Cor (v_h)) =
v_h - (\Cor \circ \Ph)(v_h) = v_h -  \Cor (v_h).
\end{align*}
Also note that $\Rlod$ is $H^1_{\kappa}$-stable, i.e.,
\begin{align*}
\| \Rlod (v) \|_{H^1_{\kappa}} \lesssim \| v \|_{H^1_{\kappa}}  \qquad \mbox{for all } v\in H^1(\Omega).
\end{align*}
The stability estimate is obtained from the $H^1$-stability of $\Ph$ (on quasi-uniform meshes, cf. \cite{BaY14}) and the $H^1_{\kappa}$-stability of $\Cor$. By standard arguments in the context of LOD spaces we have the following estimates for the error of the LOD decomposition operator.

\begin{lemma}
\label{estimates-Ritz-projec}
Let \ref{A0}--\ref{A2} be fulfilled and $h \lesssim \kappa^{-1}$. Then for every $v \in H^1(\Omega)$ and $f \in L^2(\Omega)$ such that
\begin{align*}
  a(v,w) = (f,w) \quad \text{for all } w \in H^1(\Omega)
\end{align*}
it holds
\begin{align*}
  \| v - \Rlod v \|_{L^2} \,\lesssim\, h \kappa \, \| v - \Rlod v \|_{H^1_\kappa}
  \qquad \text{and} \qquad
  \| v - \Rlod v \|_{H^1_\kappa} \,\lesssim\, h \kappa \, \| f - \Ph f \|_{L^2}.
\end{align*}
\end{lemma}

\begin{proof}
For arbitrary $v\in H^1(\Omega)$ we have by construction of $\Rlod$ that $\Ph(v-\Rlod(v))=\Ph(v-\Ph(v))=0$, hence $e_h:=v-\Rlod(v) \in W$. With Lemma \ref{lem:abeta_coercivity} we have
\begin{align*}
	\| e_h \|_{H^1_\kappa}^2 & \lesssim a(e_h,e_h) = a(v,e_h) = (f,e_h) = (f - \Ph f, e_h - \Ph e_h) \lesssim h \kappa  \| f - \Ph f \|_{L^2} \| e_h \|_{H^1_k}.
\end{align*}
This shows the $H^1_\kappa$-estimate. The $L^2$-estimate also exploits $e_h \in W$, which yields $\| e_h \|_{L^2}= \| e_h - \Ph(e_h)\|_{L^2} \lesssim h \kappa \| e_h \|_{H^1_{\kappa}}$.
\end{proof}

Now let $u \in H^2(\Omega)$ be a critical point. Then the first order condition implies that $u$ solves
\begin{align*}
a(u,w) = (f_u, w) \quad \text{for all } w \in H^1(\Omega).
\end{align*}
where $f_u = (1 - |u|^2)u$. Using Lemma \ref{lem:existence} one shows (cf. \cite{DH24}) that $f_u \in H^2(\Omega)$ with $| f_u |_{H^s} \lesssim \kappa^s$ for $s = 0,1,2$ and thus we directly conclude from Lemma \ref{estimates-Ritz-projec} the following estimates for the projection $\Rlod(u)$ and the best-approximation of $u$ in $\Vlod$.

\begin{corollary}
\label{cor-quasi-best-approx} 
Let \ref{A0}-\ref{A2} be fulfilled and $h \lesssim \kappa^{-1}$. Then for every critical point $\ucol$ of $E$ it holds
\begin{align*}
	\| u - \Rlod u \|_{H^1_\kappa} + (h \kappa)^{-1} \| u - \Rlod u \|_{L^2} \lesssim \kappa^3 h^3
\end{align*}
and
\begin{align*}
	\inf_{\vlod \in \Vlod} \| u - \vlod \|_{H^1_\kappa} \lesssim \kappa^3 h^3.
\end{align*}
\end{corollary}

\subsubsection{Proof of energy error estimates} \label{sec:proof_energy_estimates}

We now prove the estimate for the error in energy from Theorem \ref{thm:energy_error_estimate}.

\begin{proof}[Proof of Theorem \ref{thm:energy_error_estimate}]
Let us denote by $u \in H^1(\Omega)$ a global minimizer and by $\ulod \in \Vlod$ a discrete global minimizer. For any arbitrary $\vlod \in \Vlod$ we have due to the minimizing property $E(\ulod) - E(u) \le E(\vlod) - E(u)$ and it remains to estimate the right hand side. First we reformulate
\begin{align*}
	& E(\vlod) - E(u)  \\
	& = \tfrac{1}{2} a( \vlod , \vlod) - \tfrac{1}{2} a(\ucol , \ucol ) 
+ \tfrac{1}{4}  \int_\Omega  \left| 1- |\vlod|^2 \right|^2 - \left| 1- |\ucol|^2 \right|^2 \dx \\
& =  \tfrac{1}{2}  a(\ucol - \vlod ,\ucol - \vlod) 
-a(\ucol ,\ucol - \vlod)
+ \tfrac{1}{4}  \int_\Omega  \left| 1- |\vlod|^2 \right|^2 - \left| 1- |\ucol|^2 \right|^2  \dx \\
& = \tfrac{1}{2}  a(\ucol - \vlod,\ucol - \vlod) + \tfrac{1}{4}  \Re  \int_\Omega  \left| 1 - |\vlod|^2 \right|^2 - \left| 1 -|\ucol|^2 \right|^2 + 4 (|\ucol|^2 - 1 ) \ucol \overline{(\ucol - \vlod)} \dx
\end{align*}
where we used the Ginzburg-Landau equation $\langle E'(u), v \rangle = 0$ in the last step. Next a simple calculation gives the identity
\begin{eqnarray*}
\lefteqn{ \bigl( 1- |{v}|^2 \bigr)^2 -   \bigl( 1- |\ucol |^2 \bigr)^2 \,\,+\,\,
4(|\ucol|^2-1) \, \Re (\ucol \overline{(\ucol - {v} )}) 
 } \\
&\qquad \qquad =& 2(|\ucol|^2-1)|\ucol - {v}|^2 + (|\ucol - {v}|^2- 2 \Re( \ucol \overline{( \ucol - {v} )}) )^2.
\end{eqnarray*}
This yields with $|u| \le 1$ from Lemma \ref{lem:existence} and the continuity of $a(\cdot,\cdot)$ 
\begin{align*} 
& \Ecol(\vlod) - \Ecol(\ucol) \\
& = \tfrac{1}{2}  a(\ucol - \vlod,\ucol - \vlod) + \tfrac{1}{4}  \int_\Omega 2(|\ucol|^2-1)|\ucol - \vlod|^2 + (|\ucol - \vlod|^2- 2 \Re( \ucol \overline{( \ucol - \vlod )}) )^2 \dx\\
& \lesssim \| u - \vlod \|_{H^1_\kappa}^2 + \| u - \vlod \|_{L^2}^2 + \int_{\Omega} | \ucol - \vlod|^2 (| \ucol - \vlod| + |\ucol| )^2 \dx \\
& \lesssim \| u - \vlod \|_{H^1_\kappa}^2 + \| u - \vlod \|_{L^2}^2 + \| u - \vlod \|_{L^4}^4.
\end{align*}
Since $\vlod \in \Vlod$ is arbitrary we have
\begin{align} \label{eq:energy_quasibest}
	E(\vlod) - E(u) \lesssim \inf_{\vlod \in \Vlod} \left\{ \| u - \vlod \|_{H^1_\kappa}^2 + \| u - \vlod \|_{L^2}^2 + \| u - \vlod \|_{L^4}^4 \right\}.
\end{align}
The best-approximation errors in $H^1_\kappa$ and $L^2$ can be bounded by $\mathcal{O}(\kappa^3 h^3)$ with Corollary \ref{cor-quasi-best-approx}. For the $L^4$ error we first recall (cf. \cite{BrennerScott}) that for $\ell = 1,2$
\begin{align*}
	\| v - \Ph v \|_{L^4} \lesssim \kappa^\ell h^\ell \| v \|_{H^\ell_\kappa} \quad \text{for all } v \in L^4(\Omega)\cap H^\ell(\Omega). 
\end{align*}
Furthermore, $\Ph \Cor u = 0$ such that
\begin{align*}
	\| u - (\mathrm{Id} - \Cor) \Ph u \|_{L^4} & \lesssim \| u - \Ph u \|_{L^4} + \| (\mathrm{Id} - \Ph) \Cor \Ph u \|_{L^4} \\
	& \lesssim \kappa^2 h^2 \| u \|_{H^2_\kappa} + \kappa h \| \Cor \Ph u \|_{H^1_\kappa} \\
	& \lesssim \kappa^2 h^2 \| u \|_{H^2_\kappa} + \kappa h \| u - (\mathrm{Id} - \Cor) \Ph u \|_{H^1_\kappa} + \kappa h \| (\mathrm{Id} - \Ph) u \|_{H^1_\kappa} \\
	& \lesssim \kappa^2 h^2 + \kappa^4 h^4 \lesssim \kappa^2 h^2
\end{align*}
where we exploited the $H^1_\kappa$-stability of $(\mathrm{Id} - \Cor) \Ph$ through the $H^1_\kappa$-stability of $\Ph$ on quasi-uniform meshes and used Corollary \ref{cor-quasi-best-approx} in the last step. Now the claim follows from \eqref{eq:energy_quasibest} and Corollary \ref{cor-quasi-best-approx}.
\end{proof}

\subsubsection{The operator $E''(u)$ and its Ritz-projection}
\label{sec:E''(u)}

We analyze problems involving the elliptic differential operator $E^{\prime\prime}(u)\vert_{\orthiu}$. 
Recall from Lemmas~\ref{lem:coercivity} and~\ref{senv-theorem-reg-secE-pde} that $E^{\prime\prime}(u)\vert_{\orthiu}$ admits a bounded inverse on $\orthiu$, and that the corresponding solutions $z \in \orthiu$ to  
\begin{align*}
E^{\prime\prime}(u)\vert_{\orthiu} z = (f, \cdot)\vert_{\orthiu}
\end{align*}
possess $H^2$-regularity.  
Based on these results, the following corollary establishes how accurately such solutions $z$ can be approximated by a Galerkin method in the LOD space $\Vlod$.
\begin{corollary}\label{lemma:eta-hp-estimate}
Assume \ref{A0}-\ref{A2} and let {$\ucol \in H^1(\Omega)$} be a quasi-isolated minimizer of $\Ecol$.
For $f \in H^1(\Omega)$ let $z_f  \in \orthiu $ denote the corresponding unique solution to
\begin{align*}
\langle E^{\prime\prime}(\hspace{1pt}\ucol\hspace{1pt})  z_f   , { v} \rangle
\,\,=\,\, ( f, { v}) 
\qquad
\mbox{for all } v \in \orthiu.
\end{align*}
If $h \lesssim \kappa^{-1}$, then there exists a projection $\Rlodorth : \orthiu \to \Vlod \cap \orthiu$ with
\begin{eqnarray*}
\| v - \Rlodorth (v) \|_{H^1_{\kappa}}  
&\lesssim& \| v - \Rlod (v) \|_{H^1_{\kappa}}
\qquad
\mbox{and}
\qquad
\| v - \Rlodorth (v) \|_{L^2}  
\,\,\, \lesssim \,\,\,
 \| v - \Rlod (v) \|_{L^2}
\end{eqnarray*}
for all $v\in \orthiu$ and such that
\begin{align}
\label{def-eta-h-p}
\| z_f  -  \Rlodorth(z_f) \|_{L^2} + 
h \kappa \, \| z_f  -  \Rlodorth(z_f) \|_{H^1_{\kappa}}  \,\,\, \lesssim  \,\,\, \,(h \kappa)^3 \left(  \| f \|_{H^1_{\kappa}}  + \rho(\kappa) \| f\|_{L^2} \right) .
\end{align}
\end{corollary}
\begin{proof}
The desired result would follow directly if we could apply the error estimate for the projection $\Rlod$ in Lemma~\ref{estimates-Ritz-projec}. However, since $\Rlod$ maps into $\Vlod$, while we require an approximation result within $\Vlod \cap \orthiu$, an adjustment is necessary.  
This issue can be resolved by introducing a modified projection $\Rlodorth$ that enforces the phase condition.  
In the first step, we therefore follow the approach used in the proof of~\cite[Lemma~4.6]{BDH25} and define $\Rlodorth : \orthiu \to \Vlod \cap \orthiu$ by
\begin{align*}
\Rlodorth(v)
:= \Rlod(v)
- \frac{(\Rlod(v), \ci u)}{(\Rlod(\ci u), \ci u)} \, \Rlod(\ci u).
\end{align*}
By construction, $\Rlodorth$ is a projection, and we have $\Rlodorth(v) \in \Vlod \cap \orthiu$ for all $v \in \orthiu$. Since $\| \Rlod(v) \|_{H^1_{\kappa}} \lesssim \| v \|_{H^1_{\kappa}}$ for $h \lesssim \kappa^{-1}$, we obtain for any $v \in \orthiu$
\begin{eqnarray*}
\lefteqn{ \| v - \Rlodorth (v) \|_{H^1_{\kappa}}  
\,\,\, \le \,\,\,
\| v - \Rlod (v) \|_{H^1_{\kappa}}  
+ \left| \frac{(\Rlod(v) - v , \ci u ) }{ (\Rlod(\ci u) - \ci u , \ci u ) + ( \ci u , \ci u )} \right| \| \Rlod(\ci u)  \|_{H^1_{\kappa}} 
}\\
&\lesssim&
\| v - \Rlod (v) \|_{H^1_{\kappa}}  
+ \left| \frac{ \| \Rlod(v) - v \|_{L^2}  \, \| u \|_{L^2} }{ (\Rlod(\ci u) - \ci u , \ci u )_{L^2} + \| u \|_{L^2}^2 } \right| \|  u  \|_{H^1_{\kappa}}. \hspace{100pt} 
\end{eqnarray*}
Here we note that Lemma \ref{estimates-Ritz-projec}, the estimate $\|u \|_{H^1_{\kappa}} \lesssim \| u\|_{L^2}$ in Lemma \ref{lem:existence}, and $h \lesssim \kappa^{-1}$ imply
\begin{align*}
 (\Rlod(\ci u) - \ci u , \ci u )_{L^2} + \| u \|_{L^2}^2 \ge 
  \| u \|_{L^2}^2 - \| \Rlod(\ci u) - \ci u \|_{L^2}  \| u \|_{L^2} 
\gtrsim (1 - h \kappa)  \| u \|_{L^2}^2 > 0.
\end{align*}
Hence, the estimate for $\| v - \Rlodorth (v) \|_{H^1_{\kappa}}$ becomes 
\begin{eqnarray*}
\| v - \Rlodorth (v) \|_{H^1_{\kappa}}  
&\lesssim&
\| v - \Rlod (v) \|_{H^1_{\kappa}}  
+  \| v - \Rlod (v) \|_{L^2}
\,\,\, \lesssim \,\,\, \| v - \Rlod (v) \|_{H^1_{\kappa}}.   
\end{eqnarray*}
Note that we have analogously (and using Lemma~\ref{estimates-Ritz-projec}) that
\begin{eqnarray*}
\| v - \Rlodorth (v) \|_{L^2}  
&\lesssim&
 \| v - \Rlod (v) \|_{L^2}
\,\,\, \lesssim \,\,\, h\kappa\, \| v - \Rlod (v) \|_{H^1_{\kappa}}.
\end{eqnarray*}
Hence, for $h \kappa \lesssim 1$, we have
\begin{eqnarray*}
\| z_f - \Rlodorth (z_f) \|_{L^2} &\lesssim& \, h\kappa \| z_f - \Rlod (z_f) \|_{H^1_{\kappa}}
\quad
\mbox{and}
\quad
\| z_f - \Rlodorth (z_f) \|_{H^1_{\kappa}}
\,\,\, \lesssim \,\,\, \| z_f - \Rlod (z_f) \|_{H^1_{\kappa}}.
\end{eqnarray*}
Next, we recall from equation \eqref{alt-char-z} in Lemma~\ref{senv-theorem-reg-secE-pde} that $z_f \in H^2(\Omega)$ solves
\begin{align*}
a( z_f , v ) = ( g ,v) \qquad \mbox{for all } v\in H^1(\Omega), 
\end{align*}
where $g = f + (1 \!-\!  |u|^2  ) z_f - 2\, \Re  (u \overline{z_f}) \, u  -  \| u\|_{L^2}^{-2} (  f,   \ci u  )  \, \ci u$.
Consequently, we can apply Lemma~\ref{estimates-Ritz-projec} to obtain
\begin{eqnarray*}
\inf_{v_h \in \Vlod \cap \orthiu} \| z_f  -  v_h \|_{H^1_{\kappa}} 
&\lesssim&
\| z_f - \Rlod (z_f) \|_{H^1_{\kappa}}
\,\,\, \lesssim \,\,\, h \kappa \| g - \Ph g \|_{L^2} \,\,\, \lesssim \,\,\, (h \kappa)^2 \| g \|_{H^1_{\kappa}}.
\end{eqnarray*}
It remains to show that $\| g \|_{H^1_{\kappa}} \lesssim \| f \|_{H^1_{\kappa}} + \Csol \, \| f \|_{L^2}$. For this, we exploit the stability bound $\| z_f \|_{H^1_{\kappa}}
\lesssim \Csol \, \| f \|_{L^2}$ 
from Lemma \ref{senv-theorem-reg-secE-pde} together with $\Csol\gtrsim 1$,\, $|u(x)| \le 1$ and $\| u \|_{H^1_{\kappa}} \lesssim 1$ 
to obtain
\begin{align*}
\|  f + (1 \!-\!  |u|^2  ) z_f - 2\, \Re  (u \overline{z_f}) \, u \|_{H^1_{\kappa}} \,\,\,\lesssim \,\,\, \tfrac{1}{\kappa} \| \nabla f \|_{L^2} +  \rho(\kappa) \, \| f \|_{L^2}
\,\,\,\lesssim \,\,\, \rho(\kappa) \, \| f \|_{H^1_{\kappa}}.
\end{align*}
For the last term in $g$, we use $\| u\|_{H^1_{\kappa}} \lesssim \| u\|_{L^2}$ to obtain straightforwardly 
\begin{align*}
 \frac{|(  f,   \ci u  )|}{\| u\|_{L^2}^{2} }  \, \| u \|_{H^1_{\kappa}} \,\,\,\le \,\,\, \|f\|_{L^2}  \frac{ \| u \|_{H^1_{\kappa}} }{ \| u\|_{L^2} }  \,\,\,\lesssim \,\,\,  \|f\|_{L^2} .
\end{align*}
Combining the estimates for $g$ finishes the proof.
\end{proof}
The subsequent analysis relies crucially on the Ritz projection corresponding to the elliptic operator  
$E^{\prime\prime}(u)$ restricted to the subspace $\orthiu$.  
For a given function $v \in \orthiu$, we define $\RitzLOD(v) \in \Vlod \cap \orthiu$ as the unique element satisfying  
\begin{align}
\label{eq_def_Rh_rewritten}
\langle E^{\prime\prime}(u)\,\RitzLOD(v), v_h \rangle
= \langle E^{\prime\prime}(u)\,v, v_h \rangle 
\qquad \text{for all } v_h \in \Vlod \cap \orthiu.
\end{align}
The coercivity of $E^{\prime\prime}(u)$ on $\orthiu$ ensures that this projection is well defined.  
In the following, we establish error estimates for $\RitzLOD$, which will serve as a key component of the forthcoming analysis.
\begin{lemma}\label{lemma:estimate-ritz-projection}
Assume \ref{A0}--\ref{A2} and $h \lesssim \kappa^{-1}$. 
Let $u \in H^1(\Omega)$ be a quasi-isolated minimizer of $E$. 
Denote by $\RitzLOD$ the $E^{\prime\prime}(u)$--Ritz projection defined in~\eqref{eq_def_Rh_rewritten}, 
and by $\Rlod$ the LOD decomposition operator introduced in~\eqref{def-LOD-projection}. 
Then, for any $v \in \orthiu$, the following estimates hold:
\begin{align*}
\| v - \RitzLOD(v) \|_{H^1_{\kappa}}  
&\lesssim  \| v - \Rlod(v) \|_{H^1_{\kappa}} 
  \,+\,  \rho(\kappa)\,(h\kappa)^2\,\| v - \RitzLOD(v) \|_{L^2}, \\[0.5em]
\| v - \RitzLOD(v) \|_{L^2} 
&\lesssim  \big( h\kappa \,+\, \rho(\kappa)\,(h\kappa)^2 \big) 
  \| v - \RitzLOD(v) \|_{H^1_{\kappa}}.
\end{align*}
In particular, if $\rho(\kappa)\,(h\kappa)^2 \lesssim 1$, then
\begin{align*}
\| v - \RitzLOD(v) \|_{H^1_{\kappa}}  
&\lesssim  \| v - \Rlod(v) \|_{H^1_{\kappa}}.
\end{align*}
Finally, note that the term $\|v - \Rlod(v)\|_{H^1_{\kappa}}$ 
can be further estimated using Lemma~\ref{estimates-Ritz-projec} 
and Corollary~\ref{cor-quasi-best-approx} to obtain quantitative error estimates.
\end{lemma}
\begin{proof}
We use a so-called Schatz argument \cite{Sch74}. First, the G{\aa}rding inequality in Lemma \ref{senv-lemma-gaarding} yields
\begin{align*}
\|  v - \RitzLOD(v)  \|_{H^1_{\kappa}}^2
\,\,\lesssim\,\,
\langle E^{\prime\prime} (u) ( v - \RitzLOD(v) ) ,  v - \RitzLOD(v)  \rangle
+ \|  v - \RitzLOD(v)  \|_{L^2}^2.
\end{align*}
If $\xi \in \orthiu$ solves
\begin{align*}
\langle E^{\prime\prime} (u) \xi , \phi \rangle  = (  v - \RitzLOD(v)  , \phi ) \qquad \mbox{for all } \phi \in \orthiu,
\end{align*}
we obtain for arbitrary $\xi_h, v_h \in \Vlod \cap \orthiu$ that
\begin{eqnarray*}
\lefteqn{
\| \RitzLOD(v)-v \|_{H^1_{\kappa}}^2 \,\,\lesssim\,\, \langle E^{\prime\prime} (u) (\RitzLOD(v)-v +\xi ) , \RitzLOD(v)-v \rangle }\\
&=&  \langle E^{\prime\prime} (u) (v_h -v ) , \RitzLOD(v)-v \rangle + \langle E^{\prime\prime} (u) (\xi -\xi_h ) , \RitzLOD(v)-v \rangle \\
&\lesssim& \left( \|   v_h -v \|_{H^1_{\kappa}} +  \|  \xi -\xi_h \|_{H^1_{\kappa}} \right)  \|  \RitzLOD(v)-v  \|_{H^1_{\kappa}}.
\end{eqnarray*}
Dividing by $ \|  \RitzLOD(v)-v  \|_{H^1_{\kappa}}$ and since $v_h$ is arbitrary, we have
\begin{eqnarray*}
\| v- \RitzLOD(v) \|_{H^1_{\kappa}}  
&\lesssim& \inf_{v_h \in \Vlod \cap \orthiu}  \|  v - v_h \|_{H^1_{\kappa}} +  \|  \xi -\xi_h \|_{H^1_{\kappa}} .
\end{eqnarray*}
Now select $v_h=\Rlodorth(v)$ and $\xi_h= \Rlodorth(\xi)$ for the projection $\Rlodorth : \orthiu \to \Vlod \cap \orthiu$ from Corollary \ref{lemma:eta-hp-estimate}, then we have
\begin{align*}
 \|  \xi -\xi_h \|_{H^1_{\kappa}} \,\,\,\lesssim\,\,\,
(h \kappa)^2 \, \|  v - \RitzLOD(v) \|_{H^1_{\kappa}}  + \rho(\kappa) \, (h \kappa)^2 \, \| v - \RitzLOD(v) \|_{L^2}
\end{align*}
and consequently by making $(h \kappa)^2$ sufficiently small to absorb $(h \kappa)^2 \, \|  v - \RitzLOD(v) \|_{H^1_{\kappa}}$ into the left hand side we obtain
\begin{eqnarray*}
\| v - \RitzLOD(v) \|_{H^1_{\kappa}}  
&\lesssim&  \|  v - \Rlod(v)  \|_{H^1_{\kappa}} \,\,+\,\,  \rho(\kappa) \, (h \kappa)^2 \, \| v - \RitzLOD(v) \|_{L^2}.
\end{eqnarray*}
To establish the $L^2$-estimate we exploit the Galerkin orthogonality for $\RitzLOD(v) - v $, i.e.,
\begin{eqnarray*}
\langle E^{\prime\prime}(u) (\RitzLOD(v) - v), \xi_h \rangle & = & 0 \qquad \mbox{for all } \xi_h \in \Vlod \cap \orthiu,
\end{eqnarray*}
to get
\begin{eqnarray*}
\lefteqn{
\|  v - \RitzLOD(v) \|_{L^2}^2 \,\,\,\lesssim\,\,\,  \langle E^{\prime\prime}(u) \xi  ,  v - \RitzLOD(v) \rangle 
\,\,\,=\,\,\,  \langle E^{\prime\prime}(u) (\xi - \Rlodorth(\xi) ) ,  v - \RitzLOD(v) \rangle } \\
&\lesssim& \| \xi - \Rlodorth(\xi)  \|_{H^1_{\kappa}} \,\,\, \| \RitzLOD(v)-v  \|_{H^1_{\kappa}}  \\
&\lesssim& 
( h \kappa)^2 \, \|  v - \RitzLOD(v) \|_{H^1_{\kappa}}^2  + \rho(\kappa) \, (h \kappa)^2 \, \| v - \RitzLOD(v) \|_{L^2}  \| \RitzLOD(v)-v  \|_{H^1_{\kappa}}.
\end{eqnarray*}
Hence, Young's inequality implies
\begin{eqnarray*}
\|  v - \RitzLOD(v) \|_{L^2} &\lesssim&
\left( h \kappa \,  + \, \rho(\kappa) \, (h \kappa)^2 \, \right) \,  \| \RitzLOD(v)-v  \|_{H^1_{\kappa}}.
\end{eqnarray*}
\end{proof}
Having control over the Ritz-projection error $u - \RitzLOD(u)$, we next turn to towards the defect $\ulod - \RitzLOD(u)$.
\begin{lemma}
\label{lemma:H1-est-Rh}
Assume \ref{A0}--\ref{A2}, let $u \in H^1(\Omega)$ be a quasi-isolated minimizer of $E$ and $\ulod \in \Vlod \cap \orthiu$ a solution to the discrete GLE, i.e., $\langle E^{\prime}(\ulod),v_h\rangle = 0 $ for all $v_h \in \Vlod$.

Then, there exists a $\kappa$-independent constant $\eta>0$, such that if \,$(h \kappa)^2  \rho(\kappa) \le \eta$\, it holds
\begin{eqnarray*}
\| \RitzLOD(u)-\ulod \|_{H^1_{\kappa}}
&\lesssim& \rho(\kappa) \left(  \kappa^{d/4} \| u-\ulod\|_{L^4} + \kappa^{d/2} \| u-\ulod\|_{L^4}^2 \right) \| u-\ulod\|_{H^1_{\kappa}}.
\end{eqnarray*}
\end{lemma}

\begin{proof}
We proceed as in the estimate for  $u - \RitzLOD(u)$. With the G{\aa}rding inequality (Lemma \ref{senv-lemma-gaarding}) we have 
\begin{align*}
\|  \ulod - \RitzLOD(u)  \|_{H^1_{\kappa}}^2
\,\,\lesssim\,\,
\langle E^{\prime\prime} (u) ( \ulod - \RitzLOD(u) ) ,  \ulod - \RitzLOD(u)  \rangle
\,+\,
\|  \ulod - \RitzLOD(u)  \|_{L^2}^2.
\end{align*}
If $\xi \in \orthiu$ is the unique solution to
\begin{align*}
\langle E^{\prime\prime} (u) \xi , v \rangle
=
(  \ulod - \RitzLOD(u)  , v ) \qquad \mbox{for all } v \in \orthiu,
\end{align*}
we hence obtain for $\Rlodorth(\xi) \in  \Vlod \cap \orthiu$ that
\begin{eqnarray*}
\lefteqn{
\| \RitzLOD(u)-\ulod \|_{H^1_{\kappa}}^2 \,\,\lesssim\,\, \langle E^{\prime\prime} (u) (\RitzLOD(u)-\ulod+\xi) , \RitzLOD(u)-\ulod \rangle }\\
&=&  \langle E^{\prime\prime} (u) \xi , \RitzLOD(u)-\ulod \rangle + \langle E^{\prime\prime} (u) (\RitzLOD(u)-\ulod) , \RitzLOD(u)-\ulod \rangle \\
&=&   \langle E^{\prime\prime} (u) (\xi  - \Rlodorth(\xi) ) , \RitzLOD(u)-\ulod \rangle \, + \, \langle E^{\prime\prime} (u) (\RitzLOD(u)-\ulod + \Rlodorth(\xi) ) , u-\ulod \rangle \\
&\lesssim& \| \xi  - \Rlodorth(\xi) \|_{H^1_{\kappa}}  \, \| \RitzLOD(u)-\ulod \|_{H^1_{\kappa}} \, + \, \langle E^{\prime\prime} (u) (\RitzLOD(u)-\ulod + \Rlodorth(\xi) ) , u-\ulod \rangle \\
&\overset{\eqref{def-eta-h-p}}{\lesssim}& \left( 
(h \kappa)^2 \left(  \| \RitzLOD(u)-\ulod \|_{H^1_{\kappa}}  + \rho(\kappa) \| \RitzLOD(u)-\ulod \|_{L^2} \right) 
\right) \| \RitzLOD(u)-\ulod \|_{H^1_{\kappa}}  \\
&\enspace&\quad \, + \,\, \langle E^{\prime\prime} (u) (\RitzLOD(u)-\ulod + \Rlodorth(\xi) ) , u-\ulod \rangle  \\
&\lesssim& (h \kappa)^2  \rho(\kappa)   \| \RitzLOD(u)-\ulod \|_{H^1_{\kappa}}^2 
 \, + \, \langle E^{\prime\prime} (u) (\RitzLOD(u)-\ulod + \Rlodorth(\xi) ) , u-\ulod \rangle . 
\end{eqnarray*}
For $(h \kappa)^2  \rho(\kappa)$ sufficiently small, we can absorb $(h \kappa)^2  \rho(\kappa)   \| \RitzLOD(u)-\ulod \|_{H^1_{\kappa}}^2 $ into the left hand side to obtain
\begin{eqnarray}
\label{preliminary-est-Rlod-u-ulod}
\| \RitzLOD(u)-\ulod \|_{H^1_{\kappa}}^2 
&\lesssim& \langle E^{\prime\prime} (u) (\RitzLOD(u)-\ulod + \Rlodorth(\xi) ) , u-\ulod \rangle . 
\end{eqnarray}
To estimate the right hand side, we exploit the first order conditions for minimizers (i.e. the Ginzburg-Landau equations)
\begin{align*}
\langle E^{\prime}(u) , v_h \rangle \,\,=\,\, \langle E^{\prime}(\ulod) , v_h \rangle  \,\, =\,\,  0 \quad \mbox{for all } v_h\in \Vlod
\end{align*}
to write
\begin{eqnarray*}
\lefteqn{ \langle E^{\prime\prime} (u)( u-\ulod) ,v_h \rangle \,\,\,=\,\,\, \langle E^{\prime\prime} (u) u -  E^{\prime\prime} (\ulod) \ulod  ,v_h \rangle  
+ \langle ( E^{\prime\prime} (\ulod) - E^{\prime\prime} (u) ) \ulod  ,v_h \rangle } \\
&=& 2 ( |u|^2 u , v_h )  +  ( |\ulod|^2 \ulod , v_h )  - (2|u|^2 \ulod + u^2 \overline{\ulod} , v_h )  \\
&=&  ( 2 u |u-\ulod|^2 + (u-\ulod)^2 \overline{u} - |u-\ulod|^2 (u-\ulod)  , v_h ) . \hspace{90pt}
\end{eqnarray*}
Using this in \eqref{preliminary-est-Rlod-u-ulod} together with $\| u \|_{L^{\infty}} \le 1$ yields
\begin{eqnarray*}
\lefteqn{\| \RitzLOD(u)-\ulod \|_{H^1_{\kappa}}^2  }\\
&\lesssim&   \| u-\ulod\|_{L^4}^2 \, \| \RitzLOD(u) - \ulod + \Rlodorth(\xi) \|_{L^2} + \| u-\ulod\|_{L^4}^3  \, \| \RitzLOD(u) - \ulod + \Rlodorth(\xi) \|_{L^4}.
\end{eqnarray*}
Exploiting the weighted Gagliardo--Nirenberg inequality $\| v \|_{L^4} \lesssim \kappa^{d/4} \| v \|_{H^1_{\kappa}}$ (for $d=2,3$, cf. \cite{CFH25}) we obtain
\begin{eqnarray}
\label{est-ritzlod-proof-1}
\nonumber \lefteqn{ \| \RitzLOD(u)-\ulod \|_{H^1_{\kappa}}^2 
\,\,\,\lesssim\,\,\,  \left( \| u-\ulod\|_{L^4}^2 + \kappa^{d/4} \| u-\ulod\|_{L^4}^3 \right) \, \| \RitzLOD(u) - \ulod + \Rlodorth(\xi) \|_{H^1_{\kappa}} }\\
&\lesssim&  \left( \| u-\ulod\|_{L^4}^2 + \kappa^{d/4} \| u-\ulod\|_{L^4}^3 \right) \,
\left( \, \| \RitzLOD(u) - \ulod \|_{H^1_{\kappa}}  + \| \xi \|_{H^1_{\kappa}} \right) \\
\nonumber&\lesssim&  \left( \| u-\ulod\|_{L^4}^2 + \kappa^{d/4} \| u-\ulod\|_{L^4}^3 \right) \,\, (1+ \rho(\kappa))\,
\| \RitzLOD(u) - \ulod \|_{H^1_{\kappa}} .\hspace{80pt}
\end{eqnarray}
Note that we used here that $\Rlodorth$ is $H^1_{\kappa}$-stable on $\orthiu$, implying $\|\Rlodorth(\xi) \|_{H^1_{\kappa}} \lesssim  \|\xi \|_{H^1_{\kappa}}$.  This follows from the $H^1_{\kappa}$-stability of $\Rlod$ together with the estimate  $\| v - \Rlodorth(v) \|_{H^1_{\kappa}} \lesssim \| v - \Rlod(v) \|_{H^1_{\kappa}}$  for all $v \in \orthiu$, stated in Corollary~\ref{lemma:eta-hp-estimate}. As $\rho(\kappa) \gtrsim 1$, estimate \eqref{est-ritzlod-proof-1} becomes
\begin{eqnarray*}
\| \RitzLOD(u)-\ulod \|_{H^1_{\kappa}}
&\lesssim& \rho(\kappa) \left( \| u-\ulod\|_{L^4}^2 + \kappa^{d/4} \| u-\ulod\|_{L^4}^3 \right).
\end{eqnarray*}
Using once more the Gagliardo--Nirenberg inequality $\| v \|_{L^4} \lesssim \kappa^{d/4} \| v \|_{H^1_{\kappa}}$ finishes the proof.
\end{proof}

\subsubsection{Discrete minimizers close to $u$}

Lemma \ref{lemma:estimate-ritz-projection} and Lemma \ref{lemma:H1-est-Rh} provide the desired $H^1_{\kappa}$-estimate, if we can state a condition that allows us to ensure that the higher order contribution $ \rho(\kappa) \left( \kappa^{d/4} \| u- \ulod \|_{L^4} +  \kappa^{d/2} \| u- \ulod \|_{L^4}^2 \right)$ becomes sufficiently small so that
\begin{align*}
\rho(\kappa) \left( \kappa^{d/4} \| u-\ulod\|_{L^4} +  \kappa^{d/2} \| u-\ulod\|_{L^4}^2 \right)  \,  \| u-\ulod\|_{H^1_{\kappa}} 
\, \,\ll \, \,
\| u- \ulod \|_{H^1_{\kappa}}.
\end{align*}
Following the strategy developed in~\cite{CFH25}, we use a fixed-point argument to show that for each quasi-isolated minimizer~$u$ there exists a nearby discrete minimizer~$\ulod$, whose distance to~$u$ in the $L^4$- and $H^1_{\kappa}$-norms admits a $\kappa$-explicit upper bound. For this, we consider $E_{h}^\prime: H^1(\Omega) \to (H^1(\Omega))'$, a linearization of $E'$ around $u$ given by
\begin{align*}
\langle E_h^{\prime}(v) , w \rangle := \langle E^{\prime}(v) , \RitzLOD(w) \rangle + \langle E^{\prime\prime}(u) v,w-  \RitzLOD(w) \rangle
\end{align*}
for $v,w \in H^1(\Omega)$ and where we recall  $\RitzLOD$ as the Ritz projection in~\eqref{eq_def_Rh_rewritten}. Note that the prime in $E_h^\prime$ is purely part of the notation: $E_h^\prime$ does not denote the Fr\'echet derivative of a functional $E_h$. Later, the symbols $E_h^{\prime\prime}$ and $E_h^{\prime\prime\prime}$ will refer to the first and second derivatives of $E_h^\prime$, respectively. 

Using $E_h^{\prime}$, we can give an alternative characterization of discrete minimizers. The proof of the following lemma is fully analogous to that of~\cite[Lemma 5.6]{CFH25}. Since this article also serves as a review, we repeat the short proof for completeness and to make the presentation self-contained.
\begin{lemma}
\label{lemma:characterization:Ehprime}
Assume \ref{A0}--\ref{A2} and let $u \in H^1(\Omega)$ be a quasi-isolated minimizer of $E$. Then
\begin{align*}
\ulod \in \orthiu \,\quad\text{satisfies} \quad\, E_h^{\prime}(\ulod)=0 
\end{align*}
if and only if
\begin{align*}
\ulod \in \Vlod \cap \orthiu 
\qquad \text{and} \qquad
\langle E^{\prime}(\ulod), v_h \rangle = 0 \qquad \text{for all } v_h \in \Vlod \cap \orthiu.
\end{align*}
\end{lemma}

\begin{proof}
The implication ``$\Leftarrow$'' follows directly.  
For the reverse direction, assume that $\ulod \in \orthiu$ satisfies $E_h^{\prime}(\ulod)=0$.  
Then, for every $w_h \in \Vlod \cap \orthiu$ it holds that
\begin{align*}
\langle E^{\prime}(\ulod), w_h \rangle = 0.
\end{align*}
Together with $E_h^{\prime}(\ulod)=0$, this condition yields
\begin{align*}
\langle E^{\prime\prime}(u)\ulod,\, w - \RitzLOD(w) \rangle = 0
\qquad \text{for all } w \in \orthiu,
\end{align*}
where $\RitzLOD$ denotes the $\langle E^{\prime\prime}(u)\cdot,\cdot\rangle$-orthogonal projection onto $\Vlod\cap\orthiu$.  
By the coercivity of $E^{\prime\prime}(u)$ implied by the second-order sufficient condition, it follows that $\ulod$ lies in the range of this projection.  
Hence $\ulod \in \Vlod \cap \orthiu$, which completes the proof.
\end{proof}
Following the reasoning in \cite{CFH25}, our next objective is to identify a neighborhood of the minimizer $u$ in which there exists a unique discrete solution $\ulod \in \orthiu$ satisfying
\begin{align*}
E_h^{\prime}(\ulod) = 0.
\end{align*}
The size of this neighborhood will determine a quantitative, albeit pessimistic, estimate for the approximation error $\|u - \ulod\|_{H^1_{\kappa}}$, provided an additional mesh-resolution condition (to be specified later) is fulfilled.

To establish the existence of such a neighborhood, we recall several auxiliary properties of the discrete energy functional.  
All subsequent arguments and identities follow directly from the analysis in \cite{CFH25}, and are repeated here solely for completeness.

\begin{lemma}\label{lemma:properties_Eh_sec}
Under the assumptions of Lemma~\ref{lemma:characterization:Ehprime}, the following relation holds for all $v,w,z \in \orthiu$:
\begin{align*}
\langle E_h^{\prime\prime}(v) w , z \rangle
= \langle E^{\prime\prime}(v) w , \RitzLOD(z) \rangle
  + \langle E^{\prime\prime}(u) w , z - \RitzLOD(z) \rangle.
\end{align*}
In particular, for $v=u$ this simplifies to
\begin{align*}
\langle E_h^{\prime\prime}(u) w , z \rangle
= \langle E^{\prime\prime}(u) w , z \rangle,
\end{align*}
which implies the coercivity estimate
\begin{align*}
\langle E_h^{\prime\prime}(u) w , w \rangle
\;\ge\;
\rho(\kappa)^{-1}\,\|w\|_{H^1_{\kappa}}^2
\qquad \text{for all } w \in \orthiu.
\end{align*}
Moreover, the following identity holds:
\begin{align}
\label{id-u-rhu}
E^{\prime\prime}(u)\big|_{\orthiu}^{-1} E_h^{\prime}(u)\big|_{\orthiu}
\,=\, u - \RitzLOD(u),
\end{align}
where $E^{\prime\prime}(u)\big|_{\orthiu}^{-1}$ denotes the inverse of the linear operator 
$E^{\prime\prime}(u)\colon \orthiu \to (\orthiu)^{*}$,
that is, for any functional $F \in (\orthiu)^{*}$, 
the element $w_F := E^{\prime\prime}(u)\big|_{\orthiu}^{-1} F \in \orthiu$ is defined by
\begin{align*}
\langle E^{\prime\prime}(u) w_F , v \rangle = \langle F , v \rangle 
\qquad \text{for all } v \in \orthiu.
\end{align*}
\end{lemma}

\begin{proof}
All statements follow by straightforward algebraic manipulations. Only \eqref{id-u-rhu} requires verification.
Since $\langle E^{\prime\prime}(u)\cdot,\cdot\rangle$ is coercive on $\orthiu$, the inverse operator $E^{\prime\prime}(u)\big|_{\orthiu}^{-1}$ is well-defined.  
Let us denote $\phi_u := E^{\prime\prime}(u)\big|_{\orthiu}^{-1} E_h^{\prime}(u) \in \orthiu$, which satisfies,  
by construction,
\begin{align*}
\langle E^{\prime\prime}(u)\phi_u , v \rangle
= \langle E_h^{\prime}(u) , v \rangle
\qquad \text{for all } v \in \orthiu.
\end{align*}
Inserting the expression for $E_h^{\prime}(u)$ and using $E^{\prime}(u)=0$ yields
\begin{align*}
\langle E^{\prime\prime}(u)\phi_u , v \rangle
&= \langle E^{\prime\prime}(u)u , v - \RitzLOD(v) \rangle
 = \langle E^{\prime\prime}(u)(u - \RitzLOD(u)) , v - \RitzLOD(v) \rangle \\
&= \langle E^{\prime\prime}(u)(u - \RitzLOD(u)) , v \rangle.
\end{align*}
As this equality holds for all $v \in \orthiu$, and both 
$\phi_u$ and $u - \RitzLOD(u)$ belong to $\orthiu$, 
uniqueness of the Riesz representation with respect to $\langle E^{\prime\prime}(u)\cdot,\cdot\rangle$ implies $\phi_u = u - \RitzLOD(u)$, which establishes \eqref{id-u-rhu}.
\end{proof}
Continuity of $E_h^{\prime\prime}$ in a neighborhood of a minimizer $u$ is described by the following lemma.
\begin{lemma}
\label{lemma:continuity-secEh}
In the setting of Lemma \ref{lemma:characterization:Ehprime} and
Lemma \ref{lemma:properties_Eh_sec}, let $u\in H^1(\Omega)$ be a quasi-isolated minimizer.
If $\rho(\kappa)\,(h\kappa)^2 \lesssim 1$, then it holds for all $v,w,z \in \orthiu$ that
 \begin{eqnarray}
\label{eq_cont_Eh_L4}
  \langle (E^{\prime\prime}_h(u) -  E^{\prime\prime}_h(z))  w , v \rangle 
& \lesssim& \,  \left( \, \| u-z \|_{L^4} \, 
 +  \kappa^{d/4} \, \| u-z \|_{L^4}^2 \right) \| w \|_{L^4} \, \| v \|_{H^1_{\kappa}}
 \end{eqnarray}
and
 \begin{eqnarray}
\label{eq_cont_Eh_H1}
 \langle (E^{\prime\prime}_h(u) -  E^{\prime\prime}_h(z))  w , v \rangle 
& \lesssim& \,  \left( \, \kappa^{d/4} \, \| u-z \|_{L^4} \, 
 +  \kappa^{d/2} \, \| u-z \|_{L^4}^2 \right) \| w \|_{H^1_{\kappa}} \, \| v \|_{H^1_{\kappa}}.
 \end{eqnarray}
\end{lemma}
 
 \begin{proof}
 The estimates follow by Taylor expansion of $E^{\prime\prime}_h$ and the $H^1_{\kappa}$-stability of $\RitzLOD$ under the resolution condition $\rho(\kappa)\,(h\kappa)^2 \lesssim 1$ (cf. Lemma \ref{lemma:estimate-ritz-projection}). At the end, the Gagliardo--Nirenberg inequality $\| v \|_{L^4} \lesssim \kappa^{d/4} \| v \|_{H^1_{\kappa}}$ is applied. For details we refer to \cite[Lemma 5.8]{CFH25}.
\end{proof}
We now show that $E_h^{\prime}$ has a unique zero near any (quasi-isolated) minimizer $u$.
\begin{theorem}
\label{theorem:existence-of-local-discrete-minimizer}
Assume~\ref{A0}--\ref{A2}. 
Then, for each quasi-isolated minimizer $u \in H^1(\Omega)$ of $E$, 
there exist constants $\tau_{\mathrm{up}} \gtrsim 1$ and $c_{\mathrm{up}} \gtrsim 1$ 
(both independent of $\kappa$ and $h$) such that for all 
$\tau \in (0,\tau_{\mathrm{up}}]$, the following holds: If
$$
\rho(\kappa)\,(h\kappa)^2 \,\le\, c_{\mathrm{up}}
\qquad\text{and}\qquad
\kappa^{d/2}\rho(\kappa)\,(h\kappa)^3 \,\le\, c_{\mathrm{up}}\,\tau,
$$
then there exists a unique $\ulod \in \Vlod \cap \orthiu$ satisfying
\begin{align*}
\| \ulod - u \|_{H^1_{\kappa}} \,\,\le\,\, \tau \, \rho(\kappa)^{-1}
\qquad
\mbox{and}
\qquad
\| \ulod - u \|_{L^4(\Omega)} \,\, \le \, \tau \, \rho(\kappa)^{-1} \kappa^{-d/4}
\end{align*}
such that 
\[
\langle E^{\prime}(\ulod), v_h \rangle = 0 
\qquad\text{for all } v_h \in \Vlod.
\]
Moreover,
\begin{align}
\label{secEuh-coercive}
\langle E^{\prime\prime}(\ulod)v_h , v_h \rangle 
\,\gtrsim\, \rho(\kappa)^{-1}\| v_h \|_{H^1_{\kappa}}^2
\qquad\text{for all } v_h \in \Vlod \cap \orthiu,
\end{align}
i.e., $\ulod$ is a (local) minimizer of $E$ in $\Vlod$.
\end{theorem}

\begin{proof}
The proof adopts the arguments of \cite[Theorem 5.9]{CFH25}. For a given parameter $0 < \tau \le 1$ (with an upper bound to be specified later), set
\begin{align*}
\delta_{\tau,\kappa} := \tau\,\rho(\kappa)^{-1},
\end{align*}
and define the neighborhood of $u$ by
\begin{align*}
\Ku :=
\bigl\{ z \in \orthiu \,\big|\,
\|u - z\|_{H^1_{\kappa}} \le \delta_{\tau,\kappa}
\text{ and }
\|u - z\|_{L^4(\Omega)} \le \kappa^{-d/4}\,\delta_{\tau,\kappa}
\bigr\}.
\end{align*}
Since $H^1(\Omega)$ embeds continuously into $L^4(\Omega)$ for $d \le 3$, the set $\Ku$ is closed in~$\orthiu$.

We introduce the fixed-point operator $F : \Ku \to \orthiu$ defined by
\begin{align*}
    F(z) := z - \big(E^{\prime\prime}(u)\vert_{\orthiu}\big)^{-1} E_h^{\prime}(z).
\end{align*}
For sufficiently small $\tau$, we aim to prove that $F$ admits a unique fixed point $\ulod \in \Ku$ by applying the Banach fixed-point theorem.
To this end, we establish a resolution condition, depending on $h$ and $\kappa$, that guarantees $F$ is a contraction, i.e.,
\begin{align*}
    \| F(w) - F(z) \|_{H^1_{\kappa}} \le L \| w - z \|_{H^1_{\kappa}}
    \quad \text{for all } w,z \in \Ku \text{ and some } L<1,
\end{align*}
and that $F(\Ku) \subset \Ku$.

\medskip
\noindent\textbf{Step 1. Contraction property.}
For $w,z \in \Ku$, define $\gamma(s) := F(sw + (1-s)z)$. Then
\begin{align*}
F(w) - F(z) = \int_0^1 \gamma'(s)\,\mathrm{d}s.
\end{align*}
Using $F(z) = E_h^{\prime\prime}(u)\vert_{\orthiu}^{-1} \bigl(E_h^{\prime\prime}(u)z - E_h^{\prime}(z)\bigr)$, we find
\begin{align*}
\gamma'(s)
= E_h^{\prime\prime}(u)\vert_{\orthiu}^{-1}
\bigl(
E_h^{\prime\prime}(u)(w-z)
- E_h^{\prime\prime}(sw+(1-s)z)(w-z)
\bigr).
\end{align*}
As $\| E_h^{\prime\prime}(u)\vert_{\orthiu}^{-1} \mathcal{F}\|_{H^1_{\kappa}} \le \rho(\kappa) \sup\limits_{v \in \orthiu} \frac{\langle \mathcal{F}, v\rangle}{\|v\|_{H^1_{\kappa}}}$ for any bounded linear functional $\mathcal{F}$ on $\orthiu$, we conclude with Lemma~\ref{lemma:continuity-secEh}, that
\begin{eqnarray*}
\lefteqn{ \|F(w) - F(z)\|_{H^1_{\kappa}}
\,\,\,\le\,\,\,
\rho(\kappa)
\int_0^1
\sup_{v \in \orthiu}
\frac{\langle
(E_h^{\prime\prime}(u) - E_h^{\prime\prime}(sw+(1-s)z))(w-z),
v\rangle}{\|v\|_{H^1_{\kappa}}}\,\mathrm{d}s }\\
&\lesssim&
\rho(\kappa)
\sup_{s \in [0,1]}
\Bigl(
\kappa^{d/4}\|u - (sw+(1-s)z)\|_{L^4}
+ \kappa^{d/2}\|u - (sw+(1-s)z)\|_{L^4}^2
\Bigr)
\|w-z\|_{H^1_{\kappa}}.
\end{eqnarray*}
Since $w,z \in \Ku$, we have
$\|u-z\|_{L^4} \le \kappa^{-d/4}\delta_{\tau,\kappa}$ and
$\|w-z\|_{L^4} \le 2\,\kappa^{-d/4}\delta_{\tau,\kappa}$.
Thus,
\begin{align*}
\|F(w) - F(z)\|_{H^1_{\kappa}}
\lesssim
\rho(\kappa)\,\delta_{\tau,\kappa}\,\|w-z\|_{H^1_{\kappa}}
\lesssim
\tau \|w-z\|_{H^1_{\kappa}}.
\end{align*}
Hence, for sufficiently small $\tau$ (independent of $\kappa$),
\begin{align*}
\|F(w) - F(z)\|_{H^1_{\kappa}}
\le \tfrac12 \|w-z\|_{H^1_{\kappa}},
\end{align*}
so $F$ is a contraction on $\Ku$.

\medskip
\noindent\textbf{Step 2. Invariance of $\Ku$.}
We now show that $F(\Ku) \subset \Ku$.
For any $z \in \orthiu$,
\begin{align*}
F(u) - F(z)
=
\int_0^1
E_h^{\prime\prime}(u)\vert_{\orthiu}^{-1}
\bigl(
E_h^{\prime\prime}(u)(u-z)
- E_h^{\prime\prime}(su+(1-s)z)(u-z)
\bigr)\,\mathrm{d}s.
\end{align*}
Since \eqref{id-u-rhu} implies $u - F(u) = E^{\prime\prime}(u)\vert_{\orthiu}^{-1}E_h^{\prime}(u) = u - \RitzLOD(u)$,
we have
\begin{align*}
u - F(z)
&= (u - F(u)) + (F(u) - F(z))\\
&= u - \RitzLOD(u)
+ E^{\prime\prime}(u)\vert_{\orthiu}^{-1}
\int_0^1
\bigl(E_h^{\prime\prime}(u) - E_h^{\prime\prime}(su+(1-s)z)\bigr)(u-z)\,\mathrm{d}s.
\end{align*}
By Corollary~\ref{cor-quasi-best-approx} and Lemma~\ref{lemma:estimate-ritz-projection} (for $\rho(\kappa)\,(h\kappa)^2 \lesssim 1$), we have
\begin{align*}
\|u - \RitzLOD(u)\|_{H^1_{\kappa}}
\lesssim (h\kappa)^3.
\end{align*}
Applying again Lemma~\ref{lemma:continuity-secEh}, we find
\begin{eqnarray*}
\|u - F(z)\|_{H^1_{\kappa}}
&\lesssim& 
(h\kappa)^3
+ \rho(\kappa)\bigl(\kappa^{d/4}\|u-z\|_{L^4}
+ \kappa^{d/2}\|u-z\|_{L^4}^2\bigr)\|u-z\|_{H^1_{\kappa}} \\
&\lesssim&
(h\kappa)^3 + \rho(\kappa)\,\delta_{\tau,\kappa}^2
\,\,\,\lesssim\,\,\,
(h\kappa)^3 + \tau\,\delta_{\tau,\kappa}.
\end{eqnarray*}
If $h$ is chosen such that $(h\kappa)^3 \lesssim \tau\,\delta_{\tau,\kappa} \le \tau^2\rho(\kappa)^{-1}$,
then for sufficiently small $\tau$,
\begin{align*}
\|u - F(z)\|_{H^1_{\kappa}} \le \delta_{\tau,\kappa}.
\end{align*}

For the $L^4$-estimate, using the Gagliardo--Nirenberg inequality and the above bound, we have
\begin{eqnarray*}
\lefteqn{\|u - F(z)\|_{L^4}
\,\,\,\le\,\,\,
\|u - \RitzLOD(u)\|_{L^4}
+ \|
E^{\prime\prime}(u)\vert_{\orthiu}^{-1}
\int_0^1
\bigl(E_h^{\prime\prime}(u) - E_h^{\prime\prime}(su+(1-s)z)\bigr)(u-z)\,\mathrm{d}s
\|_{L^4} }\\
&\lesssim&
\kappa^{d/4}\|u - \RitzLOD(u)\|_{H^1_{\kappa}}
+ \kappa^{d/4}\rho(\kappa)
\int_0^1
\sup_{v \in \orthiu}
\frac{\langle (E_h^{\prime\prime}(u) - E_h^{\prime\prime}(su+(1-s)z))(u-z),v\rangle}{\|v\|_{H^1_{\kappa}}}\,\mathrm{d}s\\
&\lesssim&
(h\kappa)^3\kappa^{d/4} 
+ \rho(\kappa)\, \kappa^{d/4} \, \bigl(\|u-z\|_{L^4}^2 + \kappa^{d/4}\|u-z\|_{L^4}^3\bigr)
\,\,\,\lesssim\,\,\,
(h\kappa)^3\kappa^{d/4} + \tau\,\kappa^{-d/4}\,\delta_{\tau,\kappa}.
\end{eqnarray*}
If $h$ satisfies
\begin{align*}
(h\kappa)^3\kappa^{d/4}  \lesssim \tau\,\kappa^{-d/4}\delta_{\tau,\kappa}
\le \tau^2\,\kappa^{-d/4}\rho(\kappa)^{-1},
\quad \text{i.e.} \quad
\kappa^{d/2}\rho(\kappa)\,(h\kappa)^3 \lesssim \tau^2,
\end{align*}
for some sufficiently small $\tau$ (independent of $\kappa$), then $\|u - F(z)\|_{L^4} \le \kappa^{-d/4}\delta_{\tau,\kappa}$. Hence $F(\Ku) \subset \Ku$.

\medskip
\noindent\textbf{Step 3. Existence and properties of the discrete critical point.}
Since $F$ is a contraction on $\Ku$ and $F(\Ku) \subset \Ku$,
there exists a unique fixed point $\ulod \in \Ku \subset \orthiu$ such that $E_h^{\prime}(\ulod)=0$.
By Lemma~\ref{lemma:characterization:Ehprime}, we have $\ulod \in \Vlod \cap \orthiu$ and
\begin{align*}
\langle E^{\prime}(\ulod), v_h \rangle = 0
\quad \mbox{for all }\, v_h \in \Vlod \cap \orthiu.
\end{align*}
Next, we verify the coercivity of $\langle E^{\prime\prime}(\ulod)\cdot,\cdot\rangle$.
Using Lemma~\ref{lemma:continuity-secEh},
\begin{align*}
|\langle (E_h^{\prime\prime}(u) - E_h^{\prime\prime}(\ulod))v_h, v_h \rangle|
\lesssim
(\kappa^{d/4}\|u - \ulod\|_{L^4}
+ \kappa^{d/2}\|u - \ulod\|_{L^4}^2)
\|v_h\|_{H^1_{\kappa}}^2
\lesssim
\tau\,\rho(\kappa)^{-1}\|v_h\|_{H^1_{\kappa}}^2.
\end{align*}
Since
$\langle E_h^{\prime\prime}(w)v_h,v_h\rangle
= \langle E^{\prime\prime}(w)v_h,v_h\rangle$
for all $v_h \in \Vlod \cap \orthiu$ and $w \in \orthiu$, we obtain
\begin{align*}
\langle E^{\prime\prime}(\ulod)v_h,v_h\rangle
\ge
\rho(\kappa)^{-1}(1 - C\tau)\|v_h\|_{H^1_{\kappa}}^2
\gtrsim
\rho(\kappa)^{-1}\|v_h\|_{H^1_{\kappa}}^2.
\end{align*}

\medskip
\noindent\textbf{Step 4. Removing the phase constraint.}
We already know that $E^{\prime}(\ulod) = 0$ on $\Vlod \cap \orthiu$ and
that $E^{\prime\prime}(\ulod)$ is coercive there.
Hence $\ulod$ is the unique minimizer of $E$ in $\Vlod \cap \orthiu$, satisfying
\begin{align*}
\|u - \ulod\|_{H^1_{\kappa}} \le \delta_{\tau,\kappa},
\qquad
\|u - \ulod\|_{L^4} \le \kappa^{-d/4}\delta_{\tau,\kappa}.
\end{align*}
The phase constraint $v_h \in \orthiu$ does not affect the minimal energy value:
for any $v_h \in \Vlod$, there exists $\omega \in [-\pi,\pi)$ such that
$e^{i\omega}v_h \in \orthiu$ and $E(e^{\ci \omega}v_h) = E(v_h)$.
Therefore, $\ulod$ is also a local minimizer of $E$ in $\Vlod$,
and the first-order optimality condition implies
\begin{align*}
\langle E^{\prime}(\ulod), v_h \rangle = 0
\quad \mbox{for all } v_h \in \Vlod.
\end{align*}

\end{proof}

\subsubsection{Proof of $H^1_\kappa$ and $L^2$ error estimates}

\label{sec:proof_main_result}

With the existence result in Theorem \ref{theorem:existence-of-local-discrete-minimizer} and the corresponding abstract error bounds, we are now ready to establish our main result.
\begin{proof}[Proof of Theorem \ref{theorem:main-result}]
For given $0 < \tau \le \tau_{\mathrm{up}}$ (with $\tau_{\mathrm{up}}$ small enough, independent of $\kappa$ and $h$) and under the resolution conditions $\rho(\kappa)\,(h\kappa)^2 \,\le\, c_{\mathrm{up}}$ and $\kappa^{d/2}\rho(\kappa)\,(h\kappa)^3 \,\le\, c_{\mathrm{up}}\,\tau$, Theorem \ref{theorem:existence-of-local-discrete-minimizer} ensures existence of a unique minimizer  $\ulod \in \Vlod \cap \orthiu$ with $E^{\prime}(\ulod) \vert_{\Vlod}=0$ and 
\begin{align}
\label{proof:main-result-step-1-eqn}
\| u - \ulod  \|_{L^4} \,\,\le\,\,  \tau \, \rho(\kappa)^{-1} \kappa^{-d/4}
\qquad
\mbox{and}
\qquad
\| \ulod - u \|_{H^1_{\kappa}} \,\,\le\,\, \tau \, \rho(\kappa)^{-1}.
\end{align}
For this discrete minimizer $\ulod$, we use Corollary~\ref{cor-quasi-best-approx}, Lemma \ref{lemma:estimate-ritz-projection} and Lemma \ref{lemma:H1-est-Rh} to obtain
\begin{eqnarray*}
\lefteqn{ \| \ulod - u \|_{H^1_{\kappa}} \,\,\,\le\,\,\, \| u - \RitzLOD(u) \|_{H^1_{\kappa}} + \| \ulod - \RitzLOD(u) \|_{H^1_{\kappa}} } \\
&\lesssim& (h \kappa)^3  + \| \ulod - \RitzLOD(u) \|_{H^1_{\kappa}} \\
&\lesssim& (h\kappa)^3 + \rho(\kappa) \left(  \kappa^{d/4} \| u-\ulod\|_{L^4} + \kappa^{d/2} \| u-\ulod\|_{L^4}^2 \right) \| u-\ulod\|_{H^1_{\kappa}} \\
&\overset{\eqref{proof:main-result-step-1-eqn}}{\lesssim}&  (h\kappa)^3 + \tau \, \| u-\ulod\|_{H^1_{\kappa}}
\end{eqnarray*}
Hence, for sufficiently small $\tau\lesssim 1$, we can absorb $\tau \, \| u-\ulod\|_{H^1_{\kappa}}$ into the left hand side to obtain the desired $H^1_{\kappa}$-estimate $\| \ulod - u \|_{H^1_{\kappa}} \lesssim (h\kappa)^3$.

For the $L^2$-error estimate, we start with Lemma \ref{lemma:estimate-ritz-projection} to obtain
\begin{eqnarray*}
\| u - \RitzLOD(u) \|_{L^2} 
&\lesssim & \big( h\kappa \,+\, \rho(\kappa)\,(h\kappa)^2 \big) 
  \| u - \RitzLOD(u) \|_{H^1_{\kappa}}
   \,\,\,\lesssim\,\,\,  (h\kappa)^4 \,+\, \rho(\kappa)\,(h\kappa)^5.
\end{eqnarray*}
With the previous estimate for $\| \ulod - \RitzLOD(u) \|_{H^1_{\kappa}}$, this yields
\begin{eqnarray*}
\lefteqn{  \| \ulod - u \|_{L^2} 
\,\,\,\lesssim\,\,\,   \| u - \RitzLOD(u) \|_{L^2}  +   \| \ulod - \RitzLOD(u) \|_{L^2} } \\
&\lesssim&   (h\kappa)^4 \,+\, \rho(\kappa)\,(h\kappa)^5  +  \rho(\kappa) \left(  \kappa^{d/4} \| u-\ulod\|_{L^4} + \kappa^{d/2} \| u-\ulod\|_{L^4}^2 \right) \| u-\ulod\|_{H^1_{\kappa}}  \\
&\lesssim&  (h\kappa)^4 \,+\, \rho(\kappa)\,(h\kappa)^5  +  \rho(\kappa) \left(  \kappa^{d/2} + \kappa^{d/2+d/4} \| u-\ulod\|_{L^4} \right) \| u-\ulod\|_{H^1_{\kappa}}^2  \\
&\overset{\eqref{proof:main-result-step-1-eqn}}{\lesssim}&
(h\kappa)^4 \,+\, \rho(\kappa)\,(h\kappa)^5  +  \rho(\kappa) \kappa^{d/2} \| u-\ulod\|_{H^1_{\kappa}}^2 \\
&\lesssim& (h\kappa)^4 \,+\, \rho(\kappa)\,(h\kappa)^5  +  \rho(\kappa) \,\kappa^{d/2} \,(h\kappa)^6. 
\end{eqnarray*}
\end{proof}

\section{Computation of GL minimizers}
\label{sec:numerics}

Our main results from the previous section provide only estimates on the approximability of exact Ginzburg–Landau minimizers in LOD multiscale spaces. The remaining question is how to actually compute such a discrete minimizer in a numerical simulation. For this purpose, we introduce an iterative scheme based on a nonlinear conjugate Sobolev gradient descent (CSG method) which can be found in \cite{CFH25}.

\subsection{A nonlinear conjugate Sobolev gradient descent method}

The CSG method takes the form of an iteration for $n = 0,1,2,\dots$,
\begin{align*}
  u^{n+1} = u^n + \tau_n d^n,
\end{align*}
where $u^n \in H^1(\Omega)$ is the current iterate, $\tau_n > 0$ a suitably chosen step size, and $d^n \in H^1(\Omega)$ a descent direction. A very simple choice for the descent direction would be the negative of the gradient $\nabla E(u^n)$, leading to a classical discretization of the $L^2$-gradient flow. A more sophisticated choice, which we use here, involves the Sobolev gradient $\nabla_X E(u^n)$ of the energy functional with respect to a Hilbert space $X \subset H^1(\Omega)$ equipped with an inner product $(\cdot,\cdot)_X$. In our application we choose $X = \Vlod$, and the inner product $(\cdot,\cdot)_X$ will be defined below. However, the derivation of the method holds for general spaces $X \subset H^1(\Omega)$ as long as $(X, (\cdot,\cdot)_X)$ forms a Hilbert space.

The Sobolev gradient $\nabla_X E(z)$ at an arbitrary point $z \in X$ is formally defined by
\begin{align*}
  (\nabla_X E(z), w)_X = \langle E'(z), w \rangle \quad \text{for all } w \in X.
\end{align*}
In particular, it represents the direction of steepest descent with respect to the metric $(\cdot,\cdot)_X$. The metric $(\cdot,\cdot)_X$ is defined adapted to the point $z \in X$: for $v,w \in X$ we set
\begin{align*}
  (v,w)_{X,z} := a(v,w) + \big( (|z|^2 + |\bfA|^2) v,w \big)_{L^2}.
\end{align*}
We indicate the dependence on $z$ by the subscript and write $(\cdot,\cdot)_X = (\cdot,\cdot)_{X,z}$. Assuming that $z \in X$ is such that $(\cdot,\cdot)_{X,z}$ is coercive and continuous on $X$, we can calculate the Sobolev gradient as
\begin{align*}
  \nabla_{X,z} E(z) = z - \mathcal{A}_z^{-1}\big( (1 + |\bfA|^2) z \big),
\end{align*}
where, for $v \in X'$, the element $\mathcal{A}_z^{-1} v \in X$ is defined as the unique solution of
\begin{align*}
  (\mathcal{A}_z^{-1} v,w)_{X,z} = \langle v, w \rangle \quad \text{for all } w \in X.
\end{align*} 
Using the Sobolev gradient, we define the descent direction for the nonlinear CSG method as
\begin{align*}
  d^n = - \nabla_{X,u^n} E(u^n) + \beta^n d^{n-1},
\end{align*}
where $\beta^n \ge 0$ is the Polak--Ribi\`ere dissipation parameter, cf.~\cite{PR69}, defined by
\begin{align} \label{eq:PRparameter}
  \beta^n 
  := \max \Big\{0, 
    \frac{(\nabla_{X,u^n} E(u^n), \nabla_{X,u^n} E(u^n) -\nabla_{X,u^{n-1}} E(u^{n-1}))_{X,u^n}}
         {(\nabla_{X,u^{n-1}} E(u^{n-1}), \nabla_{X,u^{n-1}} E(u^{n-1}))_{X,u^{n-1}}} 
    \Big\}, \quad n \ge 1,
\end{align}
and $\beta^0 = 0$. The step size $\tau_n > 0$ is then chosen such that the energy becomes minimal along the descent direction, i.e.,
\begin{align*}
  \tau_n = \underset{\tau > 0}{\mathrm{argmin}} \, E(u^n + \tau d^n).
\end{align*}
Since $E(u^n + \tau d^n)$ is a polynomial of order four in $\tau$, the minimum can, in practice, be efficiently approximated using, for instance, a line search method. This leads to the following algorithm for the nonlinear CSG method:

\begin{definition}[nonlinear conjugate Sobolev gradient descent method] \label{def:CG}
For given $u^0 \in X$, $u^0 \neq 0$ and a tolerance $\mathrm{tol} > 0$. Set $d^0 = 0$ and perform for $n = 0,1,2,\dots$ the following steps:
\begin{enumerate}
\item \textbf{Solve} for $\delta^n \in X$ such that
\begin{align*}
	(\delta^n, w)_{X,u^n} = \big((1+|\bfA|^2) u^n, w \big)_{L^2} \quad \text{for all } w \in X.
\end{align*}
\item \textbf{Update} the descent direction
\begin{align*}
	d^n = \delta^n - u^n + \beta^n d^{n-1}
\end{align*}
with the Polak-Ribi\`ere dissipation parameter $\beta^n$ from \eqref{eq:PRparameter}.
\item \textbf{Perform} line search to determine
\begin{align*}
	\tau_n = \underset{\tau > 0}{\mathrm{argmin}} \, E(u^n + \tau d^n).
\end{align*}
\item \textbf{Iterate} according to
\begin{align*}
	u^{n+1} = u^n + \tau_n d^n.
\end{align*}
\end{enumerate} 
\end{definition}

Although a rigorous convergence analysis of the CSG-method is still open, one expects that the iteration converges to a limit $u \in X$ such that $E'(u) = 0$ in $X'$, i.e., $u \in X$ satisfies the first order condition. In practice, this limit is well approximated by the current iterate $u^n$ once a termination criterion $E(u^n) - E(u^{n-1}) < \mathrm{tol}$ is achieved for a given tolerance $\mathrm{tol} > 0$. In a post-processing step one can check the spectrum of the second derivative $E''(u^n)$ numerically to ensure that $u^n \approx u$ is actual a local minimizer of $E$ in $X$. If this is not the case one can perturb the current iterate into the direction of a neutral or negative eigenvalue and restart the iteration.

\subsection{Numerical results}

We use the nonlinear CSG method to compute GL energy minimizers in a two-dimensional model problem. As the approximation space, we employ the LOD space introduced and analyzed in Section~\ref{sec:error_analysis}. The convergence of discrete minimizers in the LOD space to exact minimizers, as stated in our main results, has already been verified numerically in~\cite{BDH25}, to which we refer for detailed numerical experiments.

In contrast, we are interested in computing GL energy minimizers with a high spatial resolution using the LOD approach for several values of the GL parameter~$\kappa$ typical for physically relevant type-II superconductors. Examples of such materials include niobium ($\kappa \sim 1$), magnesium diboride ($\kappa \sim 25$), and yttrium barium copper oxide ($\kappa \sim 100$); see~\cite{Finnetal01,McCSe65,StZw97}. Accordingly, we consider the values $\kappa = 10, 25, 50, 75, 100$.

Our focus is on global minimizers of the GL energy in the given configuration, rather than local ones. Since the CSG scheme does not necessarily converge to a global minimizer for arbitrary initial values, we initialize the scheme for several distinct initial values and compare the resulting energy levels of the obtained local minimizers to identify potential candidates for the global minimizer. The initial values are chosen empirically but are inspired by the literature in similar contexts~\cite{BWM05}. The implementation for our experiments is available as a MATLAB Code on \url{https://github.com/cdoeding/GLE_VortexStates_LOD}. \\

For the two-dimensional model problem, we set $\Omega = (0,1)^2$ and choose the magnetic vector potential
\begin{align*}
\bfA(x,y):= \sqrt{2} 
\begin{pmatrix} 
\sin (\pi x) \cos (\pi y)\\[0.3em] 
-\cos (\pi x) \sin (\pi y) 
\end{pmatrix}.
\end{align*}
We employ a uniform triangular mesh on $\Omega$ of mesh size $h = 2^{-7}$ for the LOD approximation space and refer to the literature \cite{BDH25,EHMP19,MaPeBook21} for the practical realization of the LOD. In particular, to compute a basis of $\Vlod$, we solve the corrector problems \eqref{corrctor-op-def} with $\mathcal{P}_1$-Lagrange finite elements on a fine mesh of size $h_{\mathrm{fine}} = 2^{-10}$ and use a localization to patches of $\ell = 10$ layers per element. We set $\alpha = \tfrac{1}{\sqrt{2}}(1 + \ci)$ and generate the different initial values using the functions
\begin{alignat*}{2}
	\psi_1(x,y) & = \alpha e^{-(x^2 + y^2)}, && \quad \psi_2(x,y) = \alpha (x + \ci y) e^{-(x^2 + y^2)}, \\	
	\psi_3(x,y) & = 1 - \alpha e^{-5(x^2 + y^2)}, && \quad \psi_4(x,y) = \tfrac{1}{\sqrt{\pi}} \big( \tfrac{2}{3} (x + \ci y) + \tfrac{1}{2} \big) e^{-(x^2 + y^2)}, \\
	\psi_5(x,y) & = \alpha e^{-10 \ci (x^2 + y^2)}, && \quad \psi_6(x,y) = \alpha (x + \ci y) e^{-10 \ci (x^2 + y^2)}, \\
	\psi_7(x,y) & = 1 - \alpha \sin(4\pi x)^2 \sin(4\pi y)^2, && \quad \psi_8(x,y) = \psi_6(x - \tfrac12, y - \tfrac12), \\
	\psi_9(x,y) & = \alpha (x - \tfrac12 + \ci), && \quad \psi_{10}(x,y) = \alpha,
\end{alignat*} 
which we define on the unit square $(-1,1)^2$. These functions are mapped to $\Omega$ using the transformation $\chi:(0,1) \rightarrow (-1,1)^2$, $(x,y) \mapsto \chi(x,y) = (2x-1,1y-1)$ so that the initial values for the CSG scheme are given by projection of
\begin{align*}
	\varphi_j = \psi_j \circ \chi, \quad j = 1,\dots,10
\end{align*}
into $\Vlod$. We apply the nonlinear CSG method from Definition \ref{def:CG} until a tolerance of $\mathrm{tol} = 10^{-15}$ is reached. In each iteration, we perform a golden section search on the interval $[0.1, 30]$ to find the optimal step size $\tau_n$. The interval is chosen to avoid extremely small step sizes and to keep the computational effort reasonable. Post-processing, we checked the spectrum of the second Fr\'echet derivative and obtained that all eigenvalues were positive and the second smallest eigenvalue was bounded away from zero with a spectral gap of at least $10^{-5}$. The energy levels of the obtained minimizers for each initial values and value of $\kappa$ are shown in Table \ref{tab1}. We highlight the lowest energy level for each value of $\kappa$ in bold font and plot the corresponding minimizing order parameter in Figure \ref{fig:minimizers}.

\begin{figure}[h!]
\centering
\begin{minipage}{0.19\textwidth}
\includegraphics[scale=0.16]{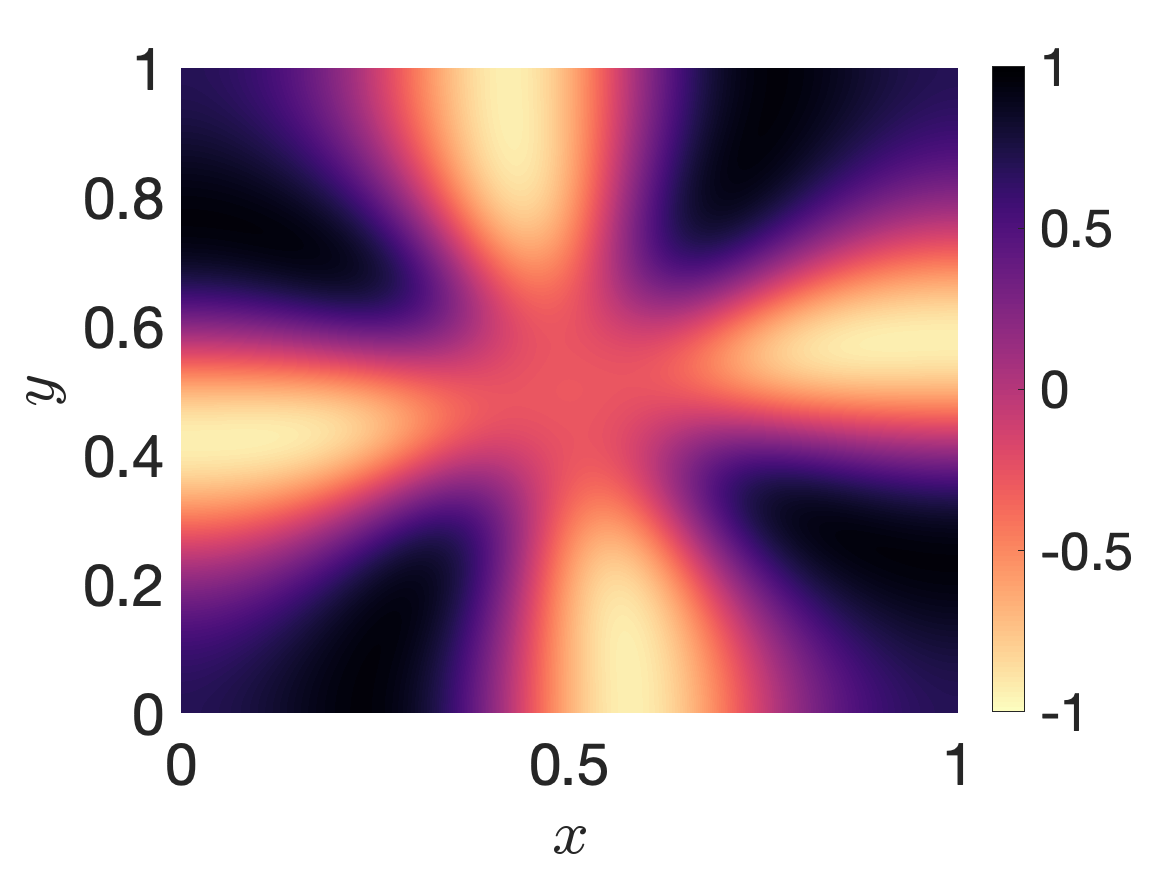}
\end{minipage}
\begin{minipage}{0.19\textwidth}
\includegraphics[scale=0.16]{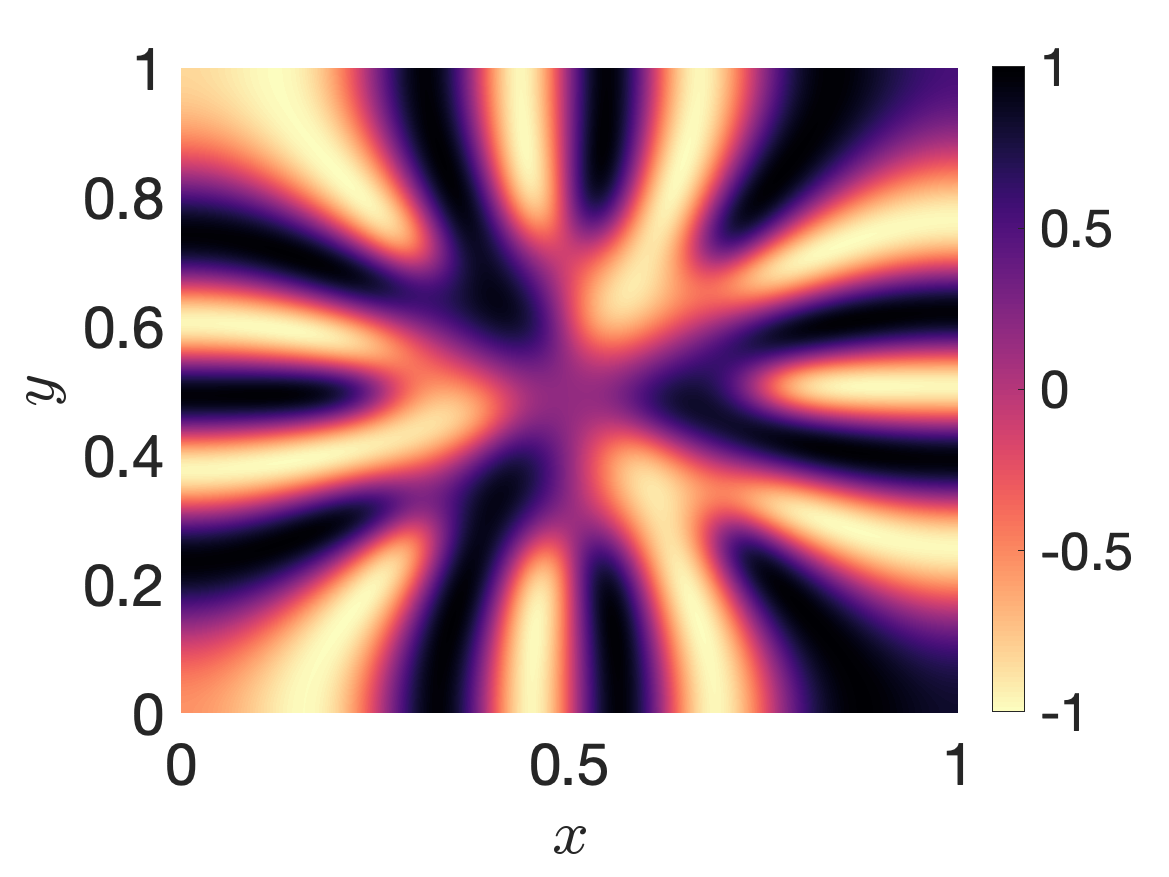}
\end{minipage}
\begin{minipage}{0.19\textwidth}
\includegraphics[scale=0.16]{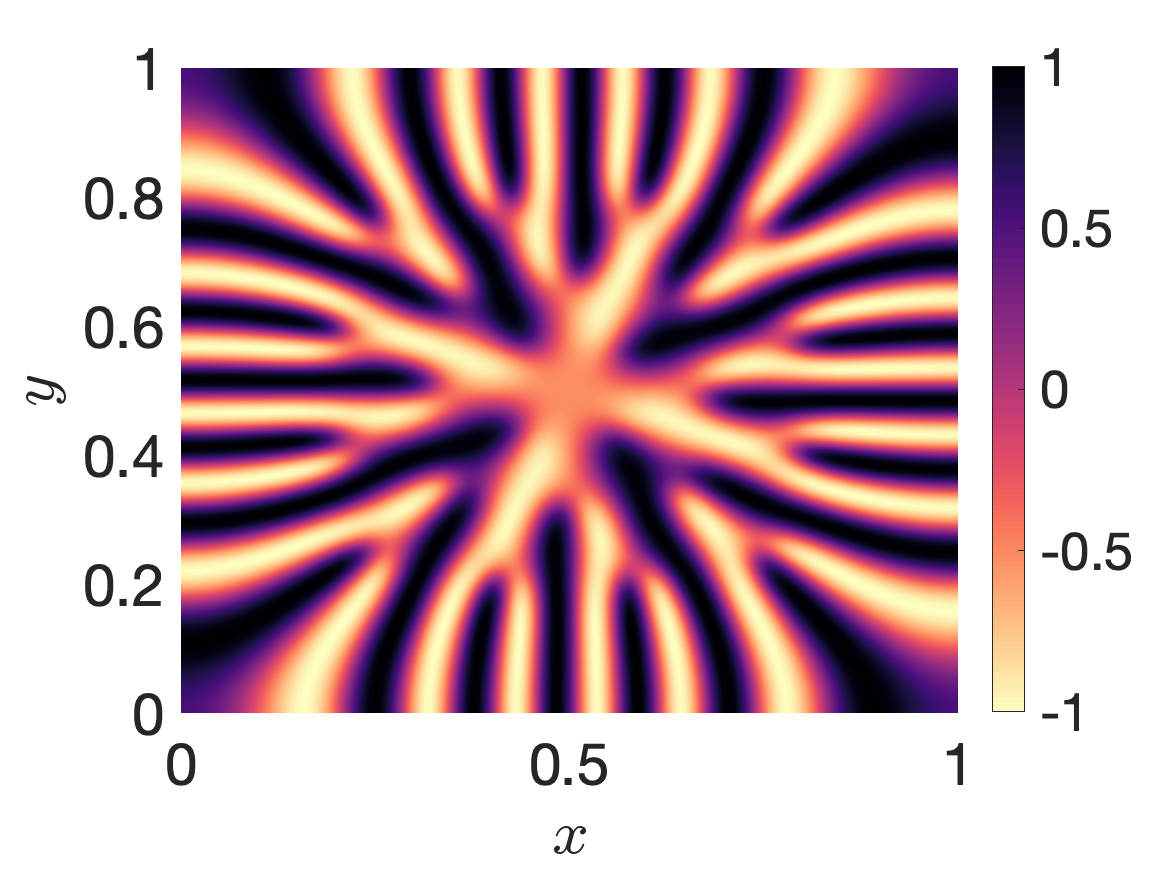}
\end{minipage}
\begin{minipage}{0.19\textwidth}
\includegraphics[scale=0.16]{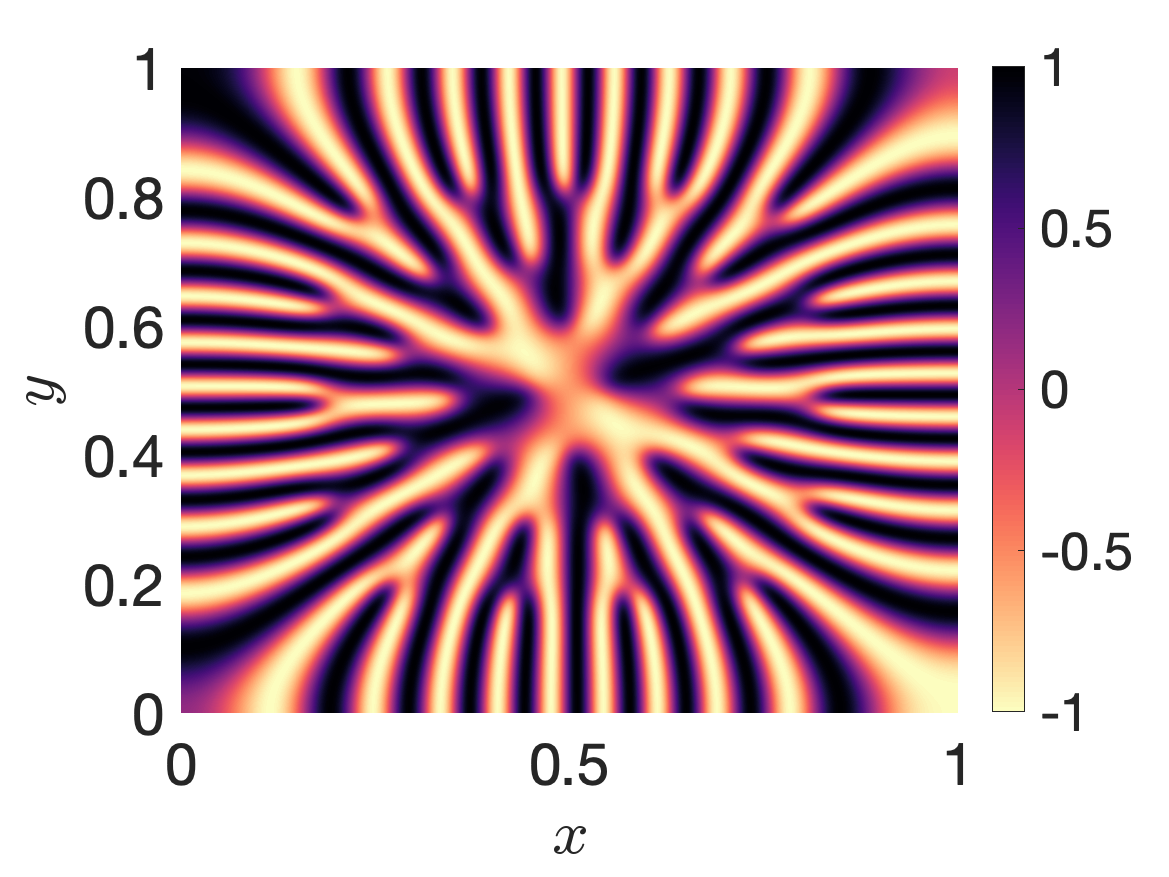}
\end{minipage}
\begin{minipage}{0.19\textwidth}
\includegraphics[scale=0.16]{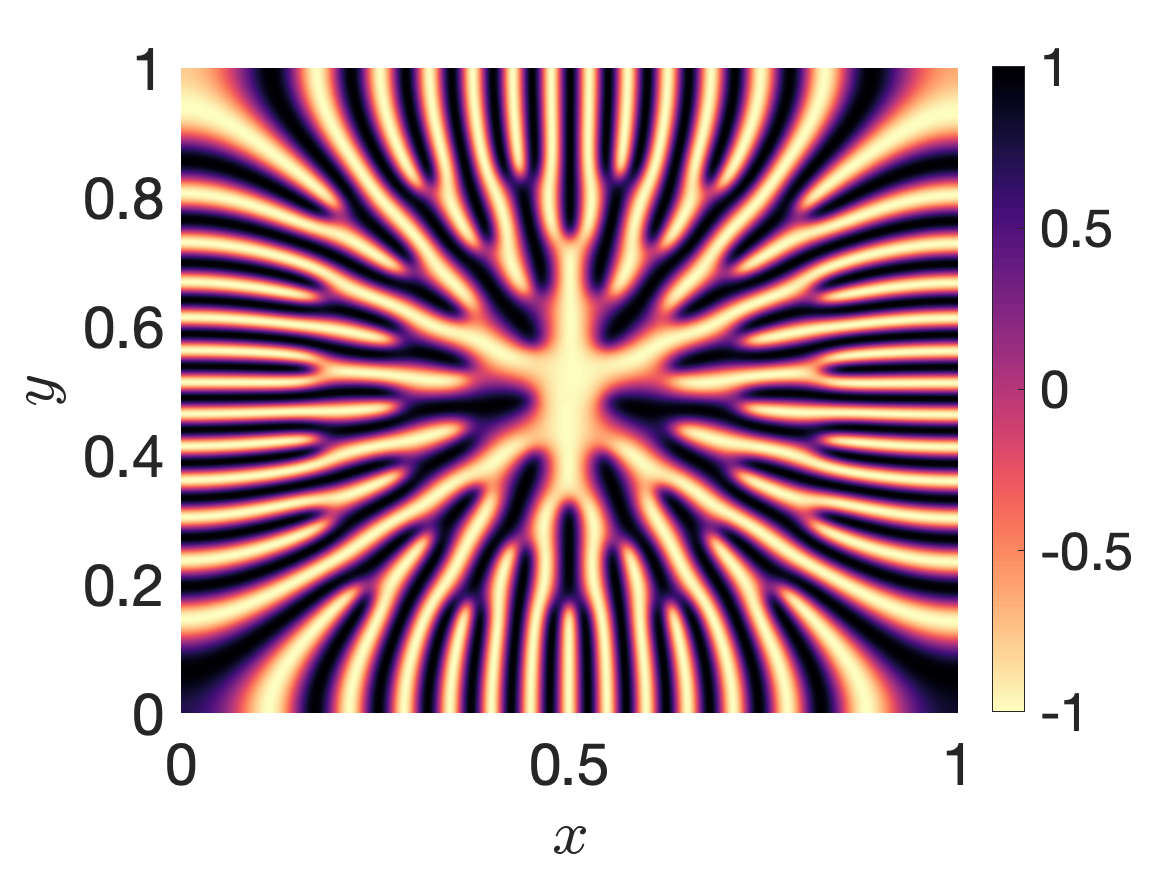}
\end{minipage} \\
\begin{minipage}{0.19\textwidth}
\includegraphics[scale=0.16]{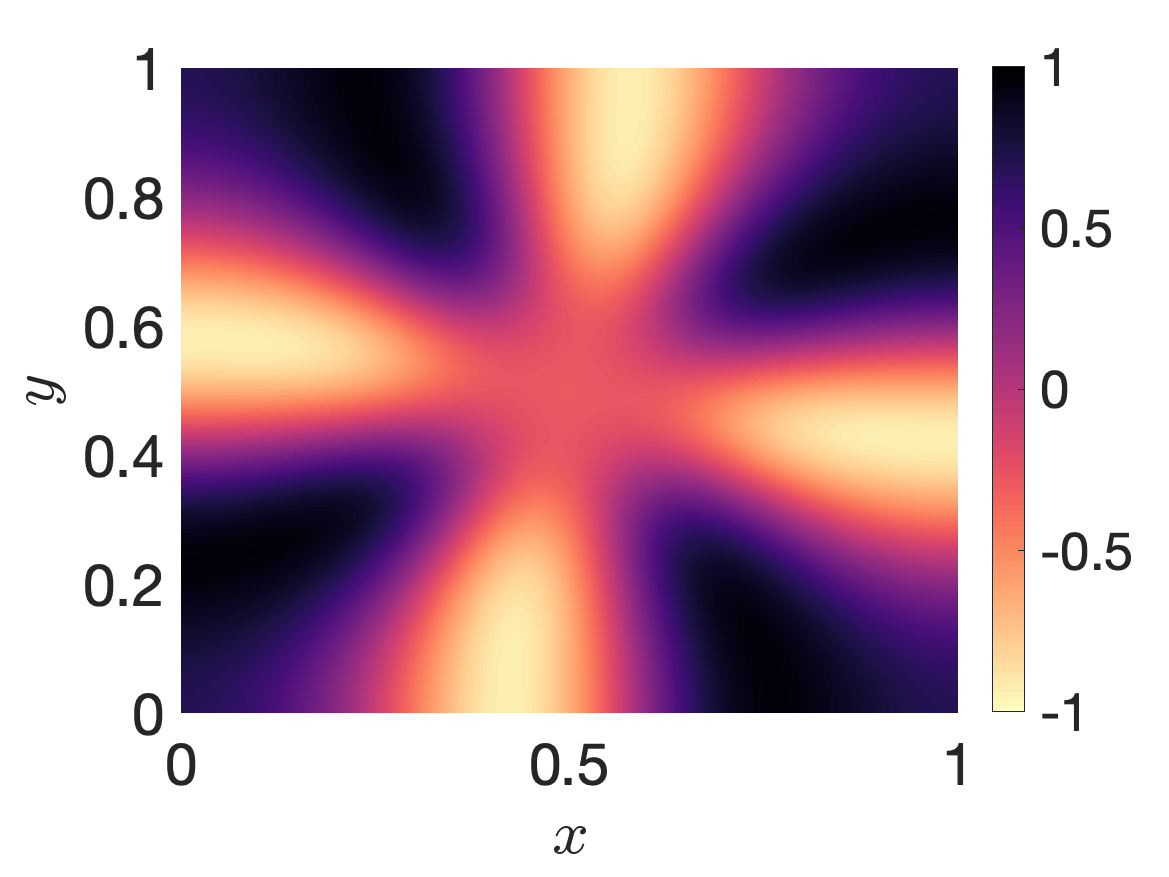}
\end{minipage}
\begin{minipage}{0.19\textwidth}
\includegraphics[scale=0.16]{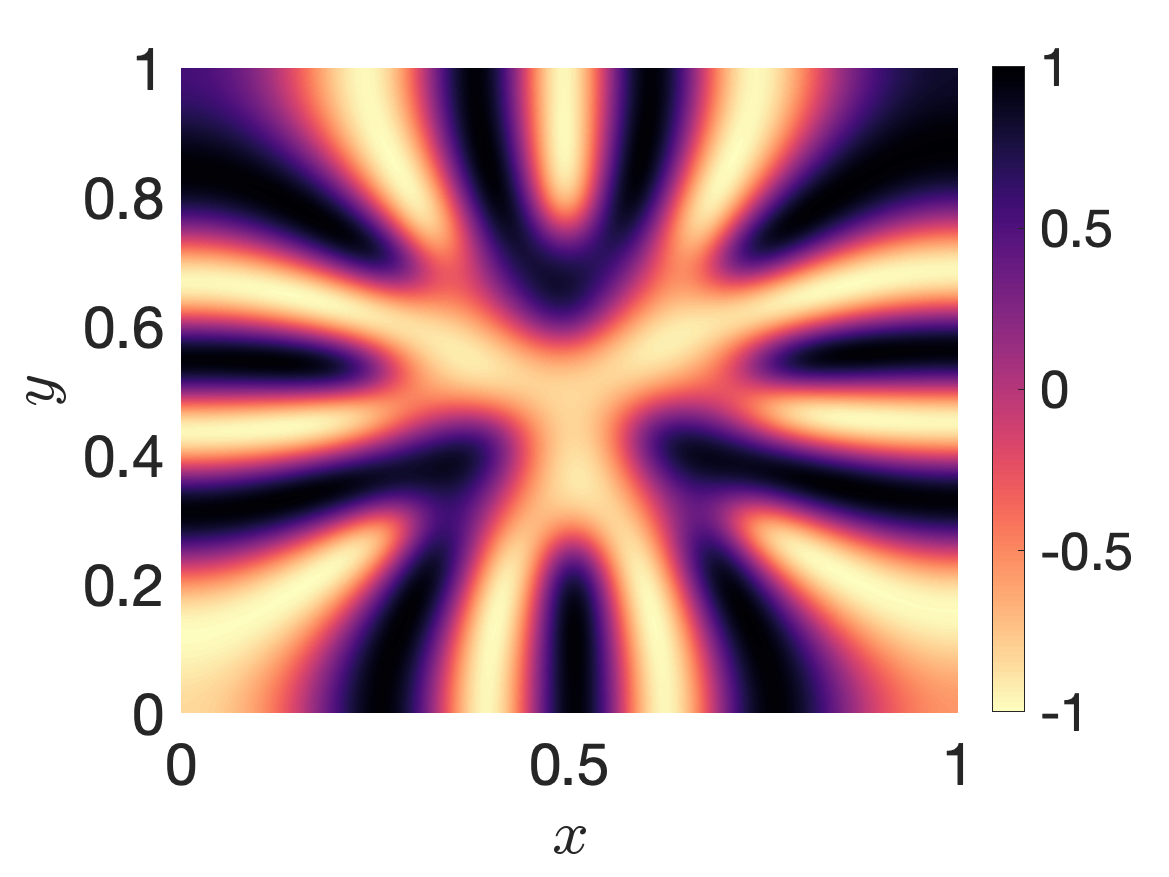}
\end{minipage}
\begin{minipage}{0.19\textwidth}
\includegraphics[scale=0.16]{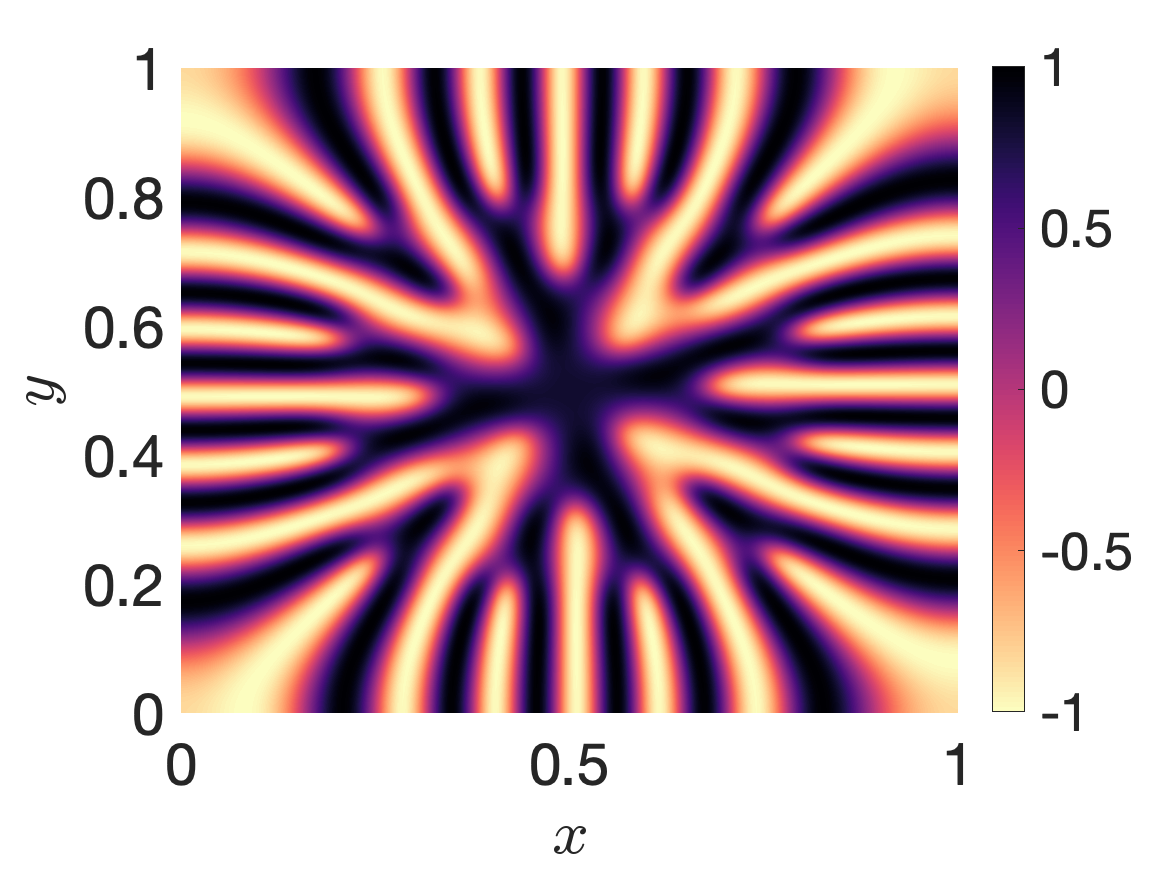}
\end{minipage}
\begin{minipage}{0.19\textwidth}
\includegraphics[scale=0.16]{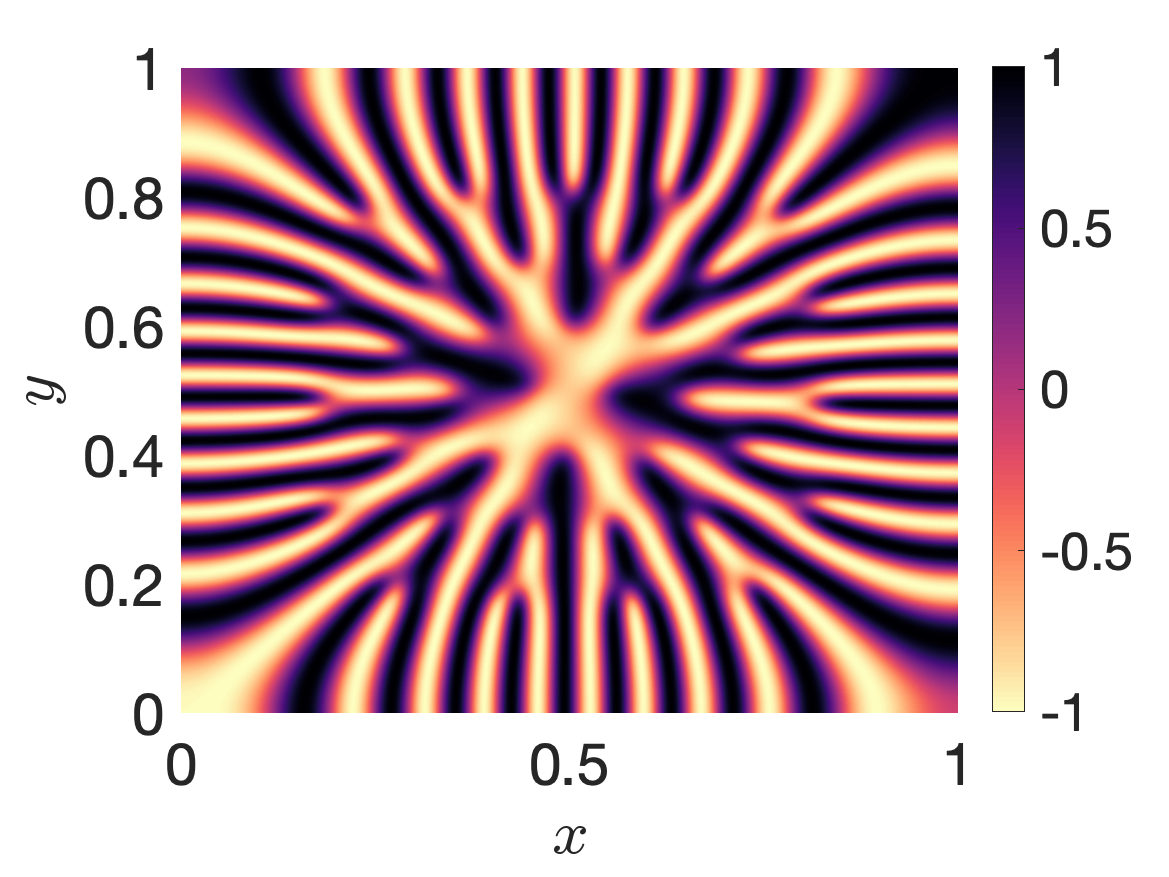}
\end{minipage}
\begin{minipage}{0.19\textwidth}
\includegraphics[scale=0.16]{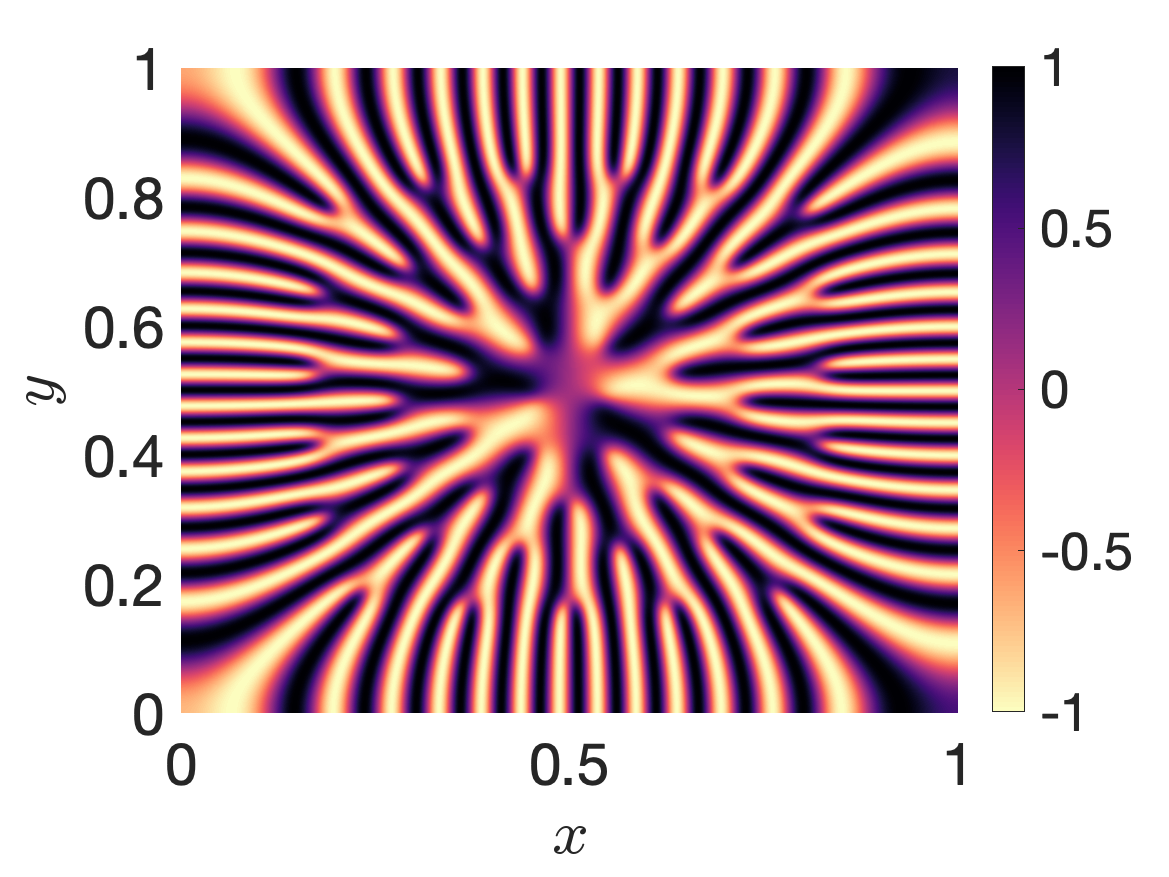}
\end{minipage} \\
\begin{minipage}{0.19\textwidth}
\includegraphics[scale=0.16]{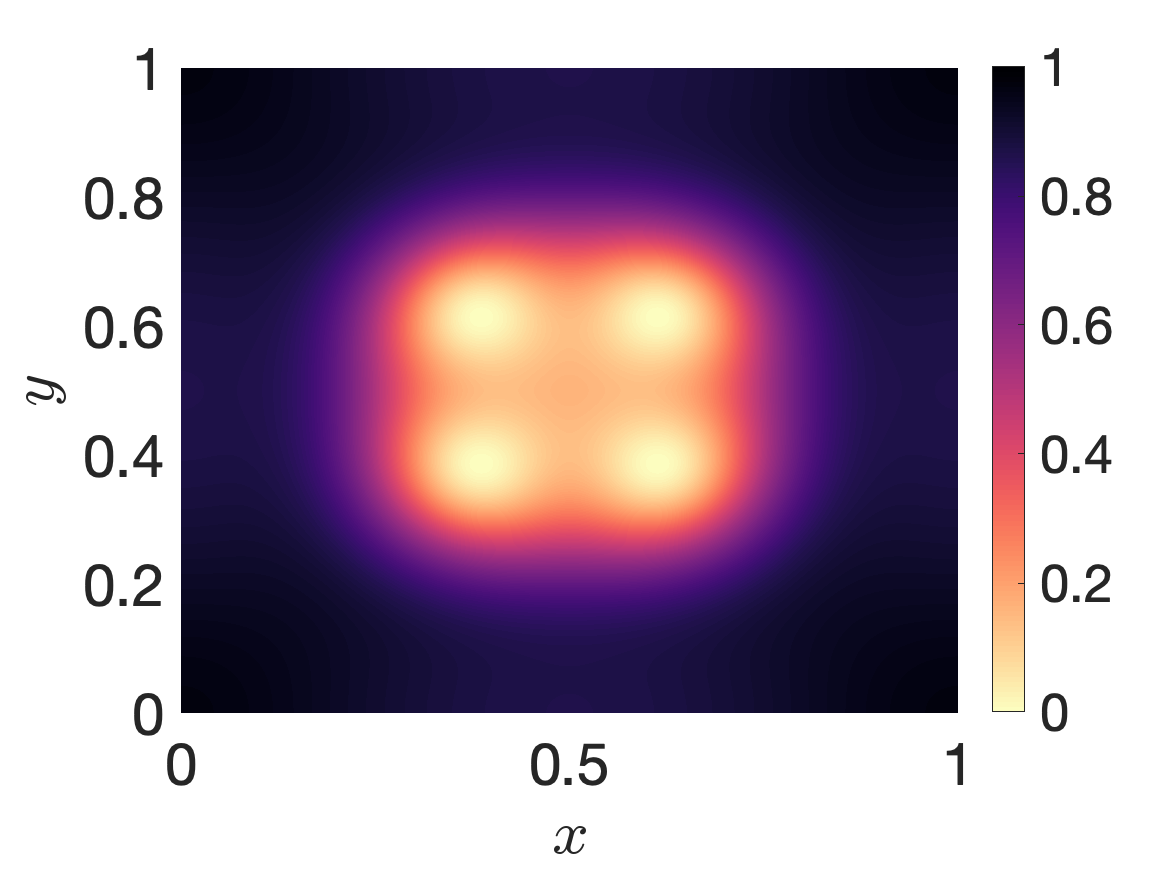}
\end{minipage}
\begin{minipage}{0.19\textwidth}
\includegraphics[scale=0.16]{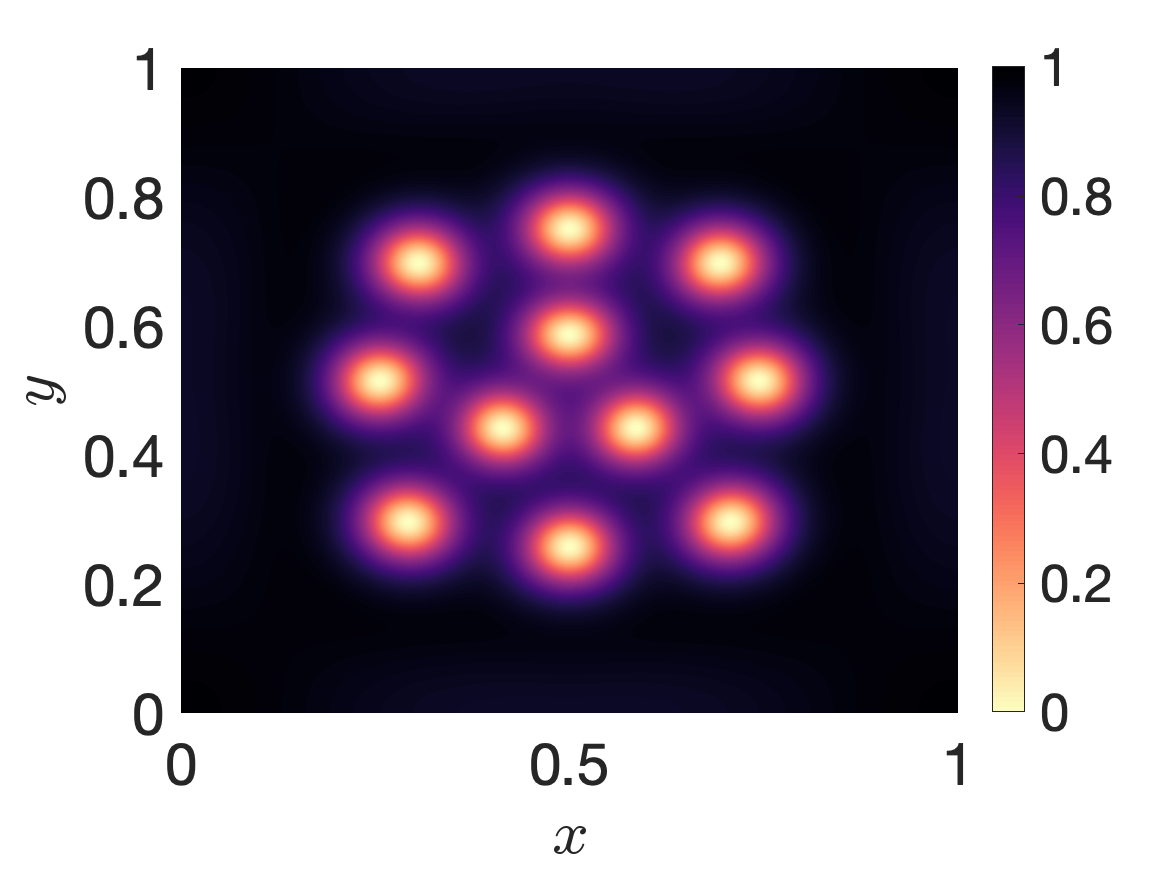}
\end{minipage}
\begin{minipage}{0.19\textwidth}
\includegraphics[scale=0.16]{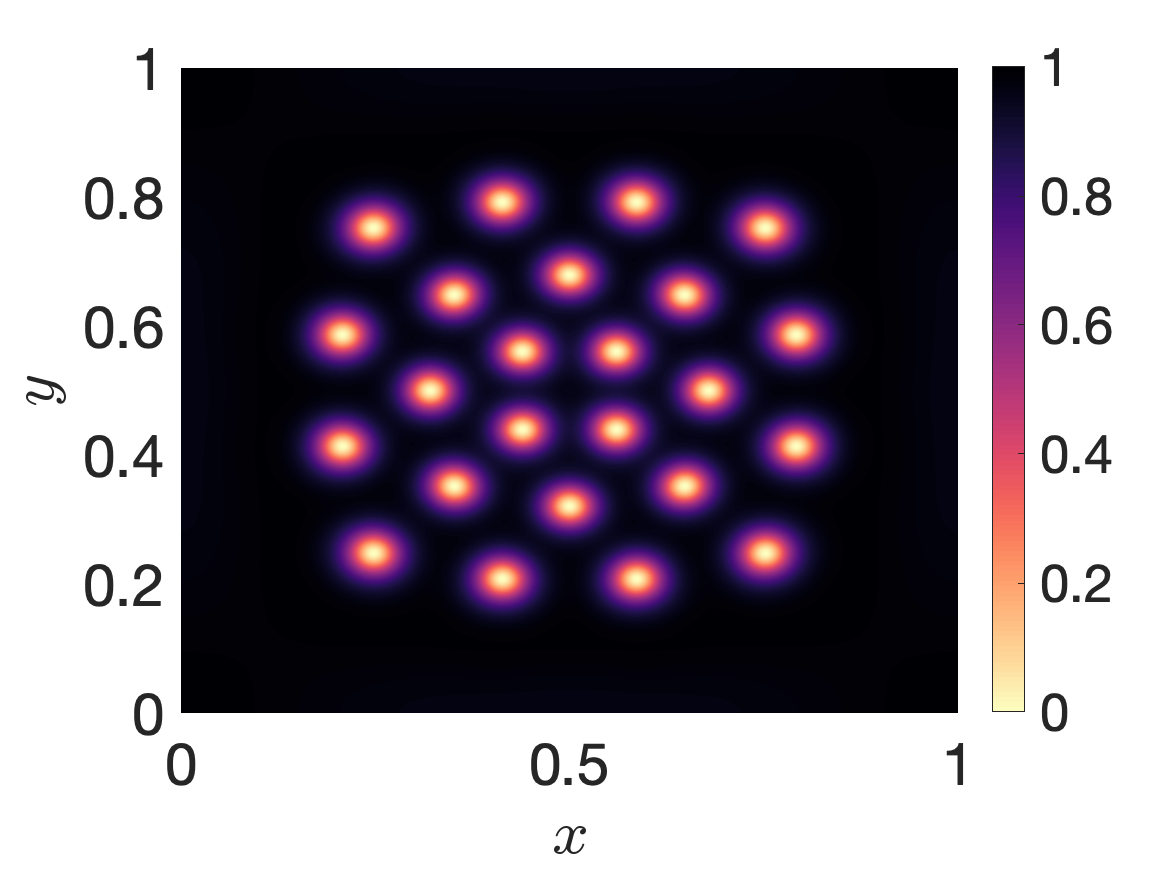}
\end{minipage}
\begin{minipage}{0.19\textwidth}
\includegraphics[scale=0.16]{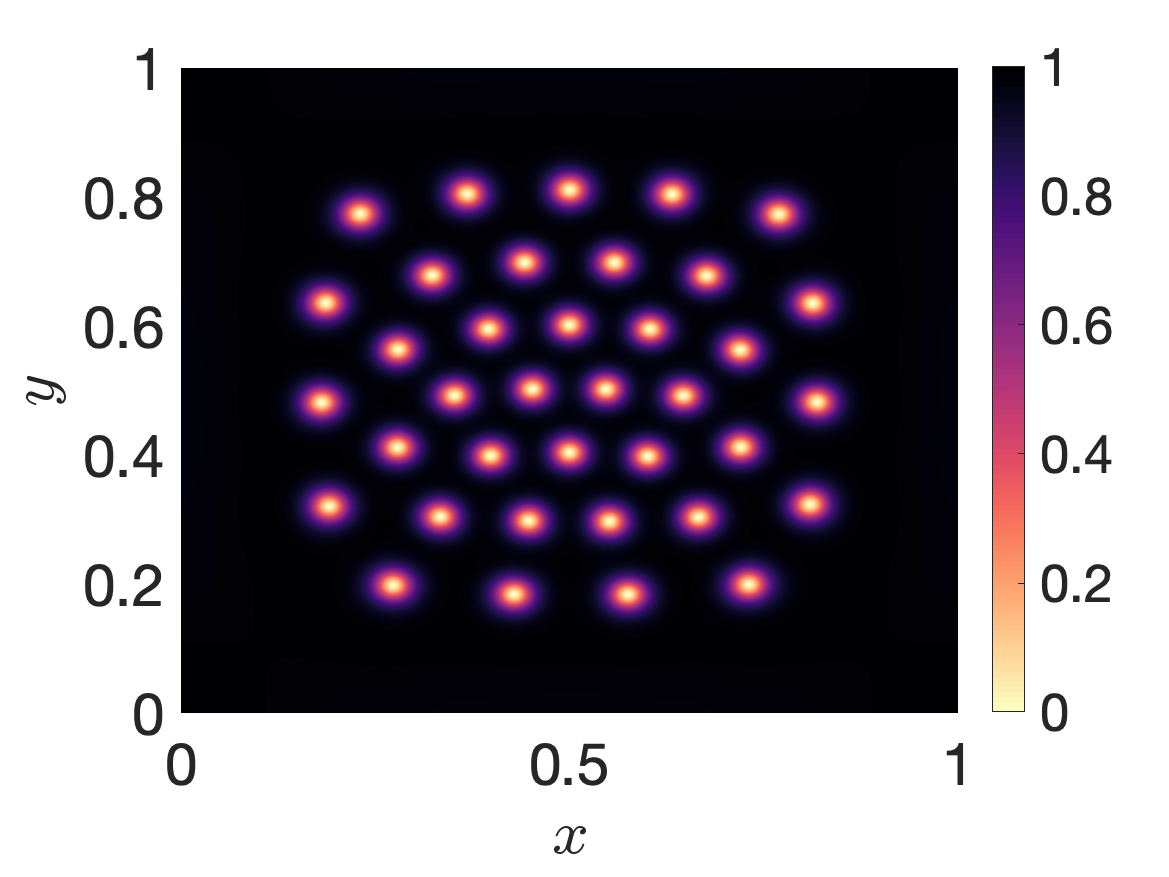}
\end{minipage}
\begin{minipage}{0.19\textwidth}
\includegraphics[scale=0.16]{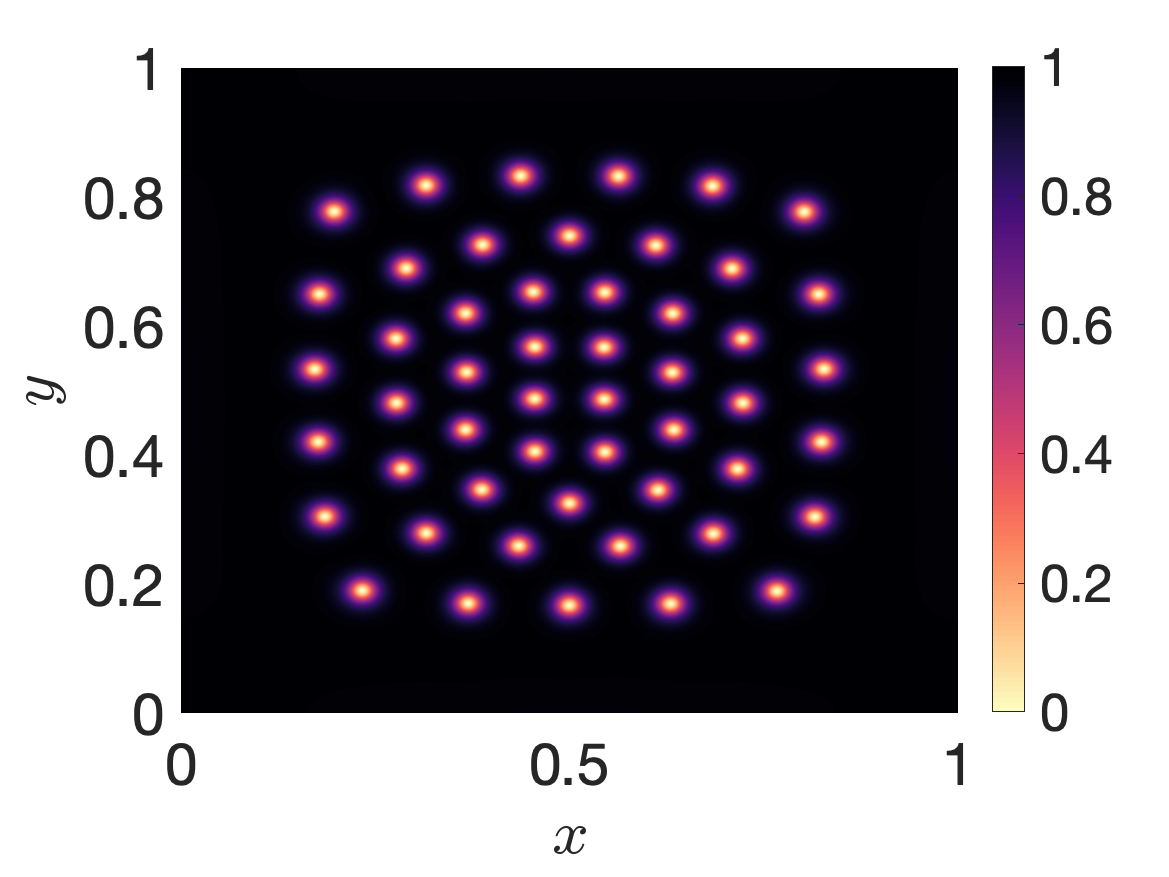}
\end{minipage}
	\caption{Real part (top line), imaginary part (middle line), and density (bottom line) of discrete minimizers $\ulod$ for $\kappa=10, 25, 50, 75, 100$ (left to right).}
	\label{fig:minimizers}
\end{figure}

\begin{table}[h!]
	\centering
	\begin{tabular}[h]{ |c|c|c|c|c|c| }
	\hline
		$\kappa$ 
		& 10 
		& 25 
		& 50
		& 75
		& 100 \\
		\hline
		$\varphi_1$ 
		& $\mathbf{0.104595899}$
		& $0.064876048$
		& $\mathbf{0.041548084}$
		& $0.036343595$
		& $0.027027260$ \\
		\hline
		$\varphi_2$ 
		& $0.118299561$
		& -
		& $0.042032619$
		& $0.038215386$
		& $0.027511422$\\
		\hline
		$\varphi_3$ 
		& $\mathbf{0.104595899}$
		& $0.064876048$
		& $0.048215444$
		& $0.036343595$
		& $0.027027260$\\
		\hline
		$\varphi_4$ 
		& $0.118299561$
		& $\mathbf{0.064533569}$
		& $\mathbf{0.041548084}$
		& $\mathbf{0.031980237}$
		& $0.027511422$\\
		\hline
		$\varphi_5$ 
		& $\mathbf{0.104595899}$
		& $0.064876048$
		& $\mathbf{0.041548084}$
		& $0.032166814$
		& $0.027027260$\\
		\hline
		$\varphi_6$ 
		& $0.118299561$
		& $\mathbf{0.064533569}$
		& $0.044209616$
		& $0.032365600$
		& $\textbf{0.026575312}$ \\
		\hline
		$\varphi_7$ 
		& $\mathbf{0.104595899}$
		& $0.064876048$
		& $0.048215444$
		& $0.032931348$
		& $0.030084660$ \\
		\hline
		$\varphi_8$ 
		& $\mathbf{0.104595899}$
		& $0.068566997$
		& $0.042723703$
		& $0.032009024$
		& $0.026581236$\\
		\hline
		$\varphi_9$ 
		& $\mathbf{0.104595899}$
		& $0.064876048$
		& $0.043236633$
		& $0.034727218$
		& $0.029023546$\\
		\hline
		$\varphi_{10}$ 
		& $\mathbf{0.104595899}$
		& $0.064876048$
		& $\mathbf{0.041548084}$
		& $0.036343595$
		& $0.036349379$\\
		\hline
		\end{tabular}
	\caption{Energy $E(\ulod)$ of (local) minimizers $\ulod$.}
	\label{tab1}
\end{table}

In view of Table~\ref{tab1}, we indeed observe multiple local minimizers whose
energy levels vary depending on the choice of the initial value. For
$\kappa = 10$, we obtain only two distinct local minimizers, whereas for larger
values of $\kappa$ we obtain up to eight different energy levels. This highlights
the sensitivity of the CSG scheme with respect to the initial value. To the
best of our knowledge, no method is known that is guaranteed to converge to the
global minimizer, nor are there results indicating how to choose the initial
value so that the CSG scheme converges to it. Our results provide no clear
indication of which initial value leads to the minimizer with the lowest
energy. Nevertheless, for the larger values $\kappa = 25, 50, 75, 100$, good
candidates appear to be $\varphi_4$ or $\varphi_6$; however, for $\kappa = 10$
these initial values do not yield the minimal energy level. There are
indications that the initial values $\varphi_j$ with $j \in
\{2,3,7,8,9\}$ may not be suitable. For example, $\varphi_2$ did not converge
for $\kappa = 25$, as the step size was reduced to its lower bound $0.1$, which
caused the energy to blow up.

Regarding the number of iterations required for the scheme to converge, we
observe a clear increase with $\kappa$: for $\kappa = 10$ roughly 50 iterations
were needed; for $\kappa = 25$ approximately 500; for $\kappa = 50$ about 1500;
for $\kappa = 75$ around 2500; and for $\kappa = 100$ about 3500 iterations. The
question of how to reliably compute global minimizers, rather than merely local
minimizers, of the Ginzburg--Landau energy remains an open problem in the
field.

Finally, we mention that we computed the spectrum of the second Fr\'echet
derivative in a post-processing step and obtained, for all local minimizers, a
strictly positive spectrum with a first eigenvalue very close to zero. This
near-zero eigenvalue corresponds to the neutral direction induced by the
symmetry of the GL energy. The spectral gap to the second-smallest eigenvalue
was at least $10^{-5}$, which is still small and, in particular, decreased with
increasing $\kappa$. This behavior reflects the dependence on the coercivity
constant $\rho(\kappa)^{-1}$ from Lemma~\ref{lem:coercivity}. Such weak
coercivity is one of the main challenges in the numerical computation of GL
energy minimizers, as it leads to slow convergence of iterative schemes.

In Figure~\ref{fig:minimizers}, we show the order parameter of the minimizers
with the lowest energy levels for the various values of $\kappa$. The Abrikosov
vortex lattice is clearly identifiable in the density of the order parameter
(bottom row). As expected, the number of vortices increases with the
Ginzburg--Landau parameter $\kappa$, while the size of the vortices
simultaneously decreases. Moreover, we observe that the structure of the order
parameter in both the real and imaginary parts (top and middle rows) is even
more intricate than the vortex lattice alone, exhibiting additional oscillations
in the complex phase. Although we do not show all local minimizers computed in
our experiments, we emphasize that the vortex patterns indeed differ across the
higher energy levels.

As demonstrated in the numerical experiments of~\cite{BDH25}, the LOD multiscale
approximation space resolves the vortex pattern with only a few degrees of
freedom. This is due to its higher order of convergence compared to standard
first-order finite elements and, in particular, to the weaker resolution
condition required to obtain meaningful approximations.

\end{document}